\documentclass[a4paper,11pt,reqno]{amsart}

\usepackage[utf8]{inputenc}
\usepackage[english]{babel}
\usepackage{amsmath}
\usepackage{amsfonts}
\usepackage{amsthm}
\usepackage{mathtools}
\usepackage{float}
\usepackage{stmaryrd}
\usepackage{graphics}
\usepackage{graphicx}
\usepackage{subfig}
\usepackage{enumitem}
\usepackage{datetime}
\usepackage{bm}
\usepackage[colorlinks=false]{hyperref}
\usepackage{color}

\newcommand{\mc}{\mathcal}
\newcommand{\eps}{\varepsilon}
\renewcommand{\d}{\,\mathrm{d}}
\newcommand{\tr}{\mathrm{tr}}
\renewcommand{\div}{\mathrm{div}}
\DeclareMathOperator*{\essup}{ess\,sup}
\DeclareMathOperator*{\argmin}{argmin}
\DeclareMathOperator{\dir}{dir}

\def\N{\mathbb{N}}
\def\R{\mathbb{R}}
\def\C{\mathbb{C}}
\renewcommand{\to}{\rightarrow}

\numberwithin{equation}{section}
\newtheorem{thm}{Theorem}[section]
\newtheorem{defi}[thm]{Definition}
\newtheorem{prop}[thm]{Proposition}
\newtheorem{lemma}[thm]{Lemma}
\newtheorem{cor}[thm]{Corollary}

\theoremstyle{definition}
\begingroup
\newtheorem{rmk}[thm]{Remark}

\endgroup

\theoremstyle{remark}

\begin{document}
	\author[F. Riva]{Filippo Riva}
	
	\title[Elastic plates with cohesive slip]{Energetic evolutions for linearly elastic plates with cohesive slip}
	
	\begin{abstract}
		A quasistatic model for a horizontally loaded thin elastic composite at small strains is studied. The composite consists of two adjacent plates whose interface behaves in a cohesive fashion with respect to the slip of the two layers. We allow for different loading-unloading regimes, distinguished by the presence of an irreversible variable describing the maximal slip reached during the evolution. Existence of energetic solutions, characterized by equilibrium conditions together with energy balance, is proved by means of a suitable version of the Minimizing Movements scheme. A crucial tool to achieve compactness of the irreversible variable are uniform estimates in H\"older spaces, obtained through the regularity theory for elliptic systems. The case in which the two plates may undergo a damage process is also considered.
	\end{abstract}
	
	\maketitle
	
	{\small
		\par
	}

	\pagenumbering{arabic}
	
	\medskip
	
	\tableofcontents

	\section*{Introduction}
	
	In recent years, cohesive zone models have attracted the interest of the mathematical community, especially due to their challenging nature and in view of diverse applications in mechanics. Unlike models of brittle rupture in solids \cite{Griffith}, in which the material instantaneously breaks as soon as a certain threshold (called toughness) is reached, cohesive models \cite{Barenblatt, Dugdale} are characterized by more gradual processes, and the progression of the rupture phenomena directly depends on the amplitude of the breaking zone itself. Cohesive behaviours are usually observed and analyzed in the framework of fracture mechanics, see \cite{CagnToad, CrisLazzOrl, DMZan, NegSca,NegVit} and references therein, or in presence of interfaces between sliding materials \cite{AleFredd1d,AleFredd2d, BonCavFredRiva}.
	
	Within the second scenario, in this paper we propose to investigate a model describing the evolution of two elastic laminates, touching along their entire surface and thus producing cohesive effects in the common interface, extending the simplified one-dimensional situation depicted in \cite{BonCavFredRiva}. The interest in such model comes from engineering applications regarding the prediction of failures in thin multilayered materials and described in \cite{AleFredd1d,AleFredd2d}, where numerical simulations have been performed in a one-dimensional and two-dimensional setting, respectively. In the quoted contributions the elastic plates also experience a damaging process during the evolution; in the current work we prefer to primarily focus on the cohesive behaviour of the interface, and thus the material is firstly assumed to be unbreakable. We however incorporate the presence of damage effects at the end of the paper, thus providing a mathematical justification of \cite{AleFredd2d}
	
	We consider two adjacent elastic plates subject to a prescribed time-dependent horizontal loading $w(t)$ acting on a portion of their boundary. The shared interface behaves cohesively with respect to the reciprocal slip between the two layers; this response may be possibly caused by roughness of the two materials or by the presence of an adhesive film gluing them. In order to distinguish among loading phases, with dissipative nature, and unloading phases, usually elastic, a fundamental role is played by an irreversible variable $\delta_h$ representing the maximum amount of slip which has taken place during the evolution, and hence called history slip.
	
	The thickness of the interface is assumed to be very small compared with the thicknesses of the two laminates, which in turn are way smaller than the surface area of the laminates themselves. Hence, the reference configuration of both elastic plates (and hence also their interface) can be described by the same planar set $\Omega\subseteq\R^2$. We however point out that, although the physical dimension of the problem is $2$, throughout the whole paper we will consider an arbitrary space dimension $n\in\N$, since from the mathematical point of view all the proposed arguments still work without changes in $\R^n$.
	
	We limit ourselves to small deformations, so that the problem can be set in the context of linearized elasticity, and the behaviour of the plates can be described by means of the two displacements $u_1,u_2\colon[0,T]\times\Omega\to \R^3$. Since the loading $w$ acts horizontally, no trasversal or bending effects actually appear during the evolution, thus the whole model can be considered as bidimensional and, as a consequence, it is not restrictive to assume that $u_1$ and $u_2$ represent in fact the in-plane displacements, hence they are valued in $\R^2$ (in the sequel in $\R^n$) instead of $\R^3$. In this way, compenetration of the two laminates is automatically avoided and no incompenetration conditions are needed. Moreover, the loading $w$ is also assumed to act slowly with respect to the internal vibrations of the body, so that inertial effects can be neglected and the model can be included in a quasistatic setting.
	
	Among the several notions of solution to quasistatic problems \cite{MielkRoubbook}, we employ the variational concept of energetic solutions, characterized by two conditions: at each time the solution minimizes the internal energy of the system, which at once balances the work done by the external loading. In the present setting, the energy is described by the functional
	\begin{equation}\label{intro:toten}
		\mc F(u_1,u_2,\delta_h)= \underbrace{\sum_{i=1}^{2}\frac 12 \int_{\Omega} \C_i(x) e(u_i(x)):e(u_i(x))\d x}_\text{bulk elastic energy of the two plates}+\underbrace{\int_{\Omega}\Phi(|u_1(x)-u_2(x)|,\delta_h(x))\d x}_\text{cohesive interfacial energy},
	\end{equation}
	where $e(u_i)$ denotes the strain tensor and $\C_i\colon \Omega\to\R^{2\times 2\times 2\times 2}$ is the fourth order elastic tensor of the $i$th layer, while $\Phi\colon [0,+\infty)^2\to[0,+\infty)$ is the cohesive energy density, which accounts for both loading ($\delta_h=|u_1-u_2|$) and unloading regimes ($\delta_h>|u_1-u_2|$). We point out that the irreversible variable $\delta_h$, which we recall models the maximal amount of occurred slip, is not an independent variable: indeed, as we will see, it explicitely depends on the displacement fields $u_1, u_2$ via the formula \eqref{historyslip}.
	
	A common procedure, which we also follow in this paper, in order to show existence of energetic solutions involves the celebrated Minimizing Movements algorithm. It consists in a time-discretization procedure followed by a recursive minimization scheme (for the displacements); the time-continuous evolution is then recovered by sending the discretization parameter $\tau$ to zero. In the current model, in order to deal with the cohesive law, this scheme is combined with a reiterated update of the discrete history slip variable, in order to preserve irreversibility.

	Compared to the one-dimensional case analyzed in \cite{BonCavFredRiva}, the major difficulty appearing in the current situation consists in finding good compactness estimates for the discrete irreversible variable, allowing for suitable convergences when the parameter $\tau$ vanishes. Indeed, a crucial tool used in the one-dimensional analysis was the embedding of the Sobolev space $H^1(a,b)$ into the space $C^\frac 12([a,b])$ of $\frac 12$-H\"older continuous functions, in order to retrieve equicontinuity of the discrete approximations. Since in higher dimensions the space $H^1(\Omega)$ is not even embedded in $L^\infty(\Omega)$, no equicontinuity properties are a priori expected anymore.
	
	We overcome this problem by exploiting the fact that the Minimizing Movements algorithm selects, at each step of the discretization process, global minimizers of the total energy \eqref{intro:toten}. The strategy is based on the computation of the Euler-Lagrange equations of the functional $\mc F$, which formally take the form
	\begin{equation}\label{intrEL}
		\begin{cases}\displaystyle
			-\div(\mathbb{C}_1 e(u_1))=-\partial_y\Phi(|u_1-u_2|,\delta_h)\frac{u_1-u_2}{|u_1-u_2|},&\text{in }\Omega,\\\displaystyle
			-\div(\mathbb{C}_2 e(u_2))=\partial_y\Phi(|u_1-u_2|,\delta_h)\frac{u_1-u_2}{|u_1-u_2|},&\text{in }\Omega.
		\end{cases}
	\end{equation}	
	The validity of the above equations, combined with Calder\'on-Zygmund $L^p$-regularity theory for elliptic systems, allows us to regain the needed H\"older estimates in order to complete the compactness argument. Anyway, a technical issue for the attainment of \eqref{intrEL} relies in the nondifferentiablity of the density $\Phi$ at the origin, indeed the presence of a kink is a crucial feature in cohesive laws \cite{Barenblatt}. The argument is thus made rigorous by introducing a suitable smooth approximation $\Phi_\eps$ of the cohesive density.

	The paper is organized as follows. In Section~\ref{sec:setting} we describe in details the mechanical model under consideration, explaining all the assumptions we require. We then present the rigorous definition of energetic solutions for the cohesive interface model, and we state our main existence result. Section~\ref{sec:regularizedenergy} is devoted to the construction of the regularized version of the cohesive density, which will be used in the Minimizing Movements algorithm. We then provide useful estimates, uniform both in the regularizing parameter $\eps$ and in the discretization parameter $\tau$, by means of energetic arguments and by employing elliptic regularity theory. These uniform bounds will be used in Section~\ref{sec:proof} in order to obtain compactness of the piecewise constant interpolant of the discrete variables. A suitable version of Helly's Selection Theorem will be needed  in order to deal with the history slip. Finally, in Section~\ref{sec:damage}, we enhance the model by considering damageable elastic plates. This framework is described by the addition of two new irreversible variables $\alpha_1,\alpha_2\colon [0,T]\times\Omega\to [0,1]$ representing the amount of damage occuring in the two laminates. We show existence of energetic solutions also for this richer model, highlighting the differences which now arise due to the presence of damage.

	\section*{Notation and preliminaries}
	
	The maximum (resp. minimum) of two extended real numbers $\alpha,\beta\in \R\cup\{\pm\infty\}$ is denoted by $\alpha\vee\beta$ (resp. $\alpha\wedge\beta$).
	
	For a positive integer $n\in \N$, we denote by $\R^{n\times n}$ and $\R^{n\times n}_{\rm sym}$ the set of real $(n\times n)$-matrices and the subset of symmetric matrices. Given a matrix $A\in \R^{n\times n}$, we write $A_{\rm sym}:=\frac12 (A+A^T)\in \R^{n\times n}_{\rm sym}$ for its symmetric part. In the case $A=\nabla u$ we adopt the standard notation $e(u)$ in place of $(\nabla u)_{\rm sym}$. The Frobenius scalar product between two matrices $A, B \in \R^{n\times n}$ is $A:B=\tr(AB^T)$, and the corresponding norm is denoted by $|A|:=\sqrt{A:A}$. The standard scalar product between vectors $a,b\in \R^n$ is denoted by $a\cdot b$ and for the euclidean norm we still write $|a|$, without risk of ambiguity. The tensor product between two vectors $a,b\in \R^n$ is the matrix $a\otimes b\in \R^{n\times n}$ defined by $(a\otimes b)_{i,j}=a_ib_j$, and the symmetric tensor product is denoted by $a\odot b:=(a\otimes b)_{\rm sym}$.
	
	We adopt standard notations for Bochner spaces and for scalar- or vector-valued Lebesgue and Sobolev spaces, while by $L^0(\Omega)^+$ we mean the space of nonnegative Lebesgue measurable functions on the (open) set $\Omega\subseteq \R^n$. Given $\alpha\in (0,1]$, by $C^{0,\alpha}(\overline{\Omega})$ and $C^{0,\alpha}(\overline{\Omega};\R^m)$ we mean, respectively, the space of scalar- and $\R^m$-valued functions which are $\alpha$-H\"older continuous (Lipschitz continuous if $\alpha=1$) in $\overline\Omega$, endowed with the norm $\Vert\cdot\Vert_{C^{0,\alpha}(\overline{\Omega})}:=\Vert\cdot\Vert_{C^0(\overline{\Omega})}+[\cdot]_{\alpha,\overline\Omega}$, where $[f]_{\alpha,\overline\Omega}:=\sup\limits_{\substack{{x,y\in\overline\Omega}\\x\neq y}}\frac{|f(x)-f(y)|}{|x-y|^{\alpha}}$. In order to lighten the notation, we write the same symbol for the norms in $C^{0,\alpha}(\overline{\Omega})$ and in $C^{0,\alpha}(\overline{\Omega};\R^m)$; the meaning will be clear from the context. We do the same for norms in Lebesge or Sobolev spaces. We finally denote with $C^{0,\alpha}_{\rm loc}({\Omega})$ (resp. $C^{0,\alpha}_{\rm loc}({\Omega};\R^m)$) the space of functions belonging to $C^{0,\alpha}(\overline{\Omega'})$ (resp. $C^{0,\alpha}(\overline{\Omega'};\R^m)$) for all $\Omega'\subset\subset \Omega$, i.e. such that the closure of $\Omega'$ is still a subset of $\Omega$.
	
	Given a normed space $(X,\|\cdot\|_X)$, with the symbol $B([a,b];X)$ we mean the space of everywhere defined functions $f\colon [a,b]\to X$ which are bounded in $X$, namely $\sup\limits_{t\in [a,b]}\|f(t)\|_X<+\infty$. The spaces of absolutely continuous functions and functions of bounded variation from $[a,b]$ to $X$ are instead denoted by $AC([a,b];X)$ and $BV([a,b];X)$, respectively. We quote for instance the Appendix of \cite{Brez} for more details on these functional spaces.
	
	For ease of reading we recall here the well-known Sobolev Embedding Theorem and the Korn-Poincar\'e inequality \cite{KondratevOle, Pompe}:
	\begin{thm}[\textbf{Sobolev Embedding}]\label{SobEmb}
		Fix $n,m\in \N$, let $\Omega\subseteq\R^n$ be an open, bounded, connected set with Lipschitz boundary and let $p\in [1,+\infty]$.
		\begin{itemize}
			\item[(a)] If $p<n$, then $W^{1,p}(\Omega;\R^m)\hookrightarrow L^q(\Omega;\R^m)$ for all $q\in[1,p^*]$, with $p^*:=np/(n-p)$;
			\item[(b)] If $p=n$, then $W^{1,p}(\Omega;\R^m)\hookrightarrow L^q(\Omega;\R^m)$ for all $q\in[1,+\infty)$;
			\item[(c)] If $p>n$, then $W^{1,p}(\Omega;\R^m)\hookrightarrow C^{0,\alpha}(\overline{\Omega};\R^m)$ for all $\alpha\in(0,1-n/p]$.
		\end{itemize}
	All the above inclusions are continuous.
	\end{thm}

\begin{prop}[\textbf{Korn-Poincar\'e inequality}]\label{prop:korn}
	Fix $n\in \N$, let $\Omega\subseteq\R^n$ be an open, bounded, connected set with Lipschitz boundary and let $\partial_D\Omega$ be a subset of $\partial\Omega$ with positive Hausdorff measure $\mc H^{n-1}(\partial_D\Omega)>0$. Fix $p\in (1,+\infty)$. Then there exists a constant $K_p>0$ such that
	\begin{equation}\label{kornpoincare}
		\Vert u\Vert_{W^{1,p}(\Omega)}\le K_p\Vert e(u)\Vert_{L^p(\Omega)}, \quad\text{ for all }u\in W^{1,p}(\Omega;\R^n)\text{ with } u=0 \text{ on }\partial_D\Omega.
	\end{equation}
\end{prop}
	
	\section{Setting and main result}\label{sec:setting}
	We consider a composite material made of two adjacent elastic layers, whose reference configuration is represented by a set $\Omega\subseteq \R^n$, with $n\in\N$ (we recall that the physical dimension is $n=2$), which we assume to satisfy the following property:
	\begin{equation}\label{Omega}
		\text{$\Omega$ is bilipschitz diffeomorphic to the open unit cube in $\R^n$.}
	\end{equation}
	This request is needed for a technical reason, namely the regularity result stated in Theorem~\ref{thmregularity}. In particular, we observe that \eqref{Omega} implies
	\begin{equation*}
		\text{$\Omega$ is open, bounded, simply connected, with Lipschitz boundary.}
	\end{equation*}
	\subsection{Elastic energy}
	Both layers of the material are assumed to be linearly elastic, so that their behaviour can be described by the two displacements $u_i$, for $i=1,2$. Since in the considered model the laminate will stretch only in the horizontal components, due to the effects of the horizontal loading (see Section~\ref{subsec:loading}), we may assume that $u_i$ actually represent the in-plane displacements, and so they are valued in $\R^n$ instead of $\R^{n+1}$. In particular, compenetration of the two laminates is automatically avoided and no incompenetration conditions are needed. Denoting with the bold letter $\bm u$ the pair $(u_1,u_2)$, we thus introduce the total bulk elastic energy $\mc E\colon H^1(\Omega;\R^n)^2\to [0,+\infty)$ given by
	\begin{equation}\label{elasticenergy}
		\mc E(\bm{u}):=\sum_{i=1}^{2}\frac 12 \int_{\Omega} \C_i(x) e(u_i(x)):e(u_i(x))\d x,
	\end{equation}
	where $e(u)$ is the symmetric gradient (strain tensor) and $\C_i\colon \Omega\to\R^{n\times n\times n\times n}$ is the fourth order elastic (or stiffness) tensor of the $i$th layer. For $i=1,2$ we assume that
	\begin{enumerate}[label=\textup{(C\arabic*)}, start=1]
		\item \label{hyp:C1} $\C_i$ is uniformly continuous with modulus of continuity $\omega_i$,
	\end{enumerate}
	together with the usual assumptions in linearized elasticity
	\begin{enumerate}[label=\textup{(C\arabic*)}, start=2]	
		\item \label{hyp:C2} $\C_i(x)A\in \R^{n\times n}_{\rm sym}$ for all $x\in\Omega$ and $A\in \R^{n\times n}$;
		\item \label{hyp:C3} $\C_i(x)A=\C_i(x)A_{\rm sym}$ for all $x\in\Omega$ and $A\in \R^{n\times n}$;
		\item \label{hyp:C4} $\C_i(x)A:B=\C_i(x)B:A$ for all $x\in\Omega$ and $A,B\in \R^{n\times n}$ (symmetry);
		\item \label{hyp:C5} $\C_i(x)A:A\ge c_i \vert A_{\rm sym}\vert^2$ for some $c_i>0$ and for all $x\in\Omega$ and $A\in \R^{n\times n}$ (coercivity).
	\end{enumerate}
	We notice that the coercivity condition \ref{hyp:C5} automatically implies the so-called strict Legendre-Hadamard condition
	\begin{equation}\label{LegendreHadamard}
			\C_i(x)(a\otimes b):(a\otimes b)\ge \frac{c_i}{2} \vert a\otimes b\vert^2,\qquad \text{ for all $x\in\Omega$ and $a,b\in \R^{n}$}.
	\end{equation}
	Indeed, \eqref{LegendreHadamard} follows from \ref{hyp:C5} by means of the easy equality
	\begin{equation*}
		2\vert a\odot b\vert^2=\vert a\otimes b\vert^2+(a\cdot b)^2.
	\end{equation*}
	
	For more insights on the Legendre-Hadamard condition we quote \cite[end of Chapter 5]{Ciarlet} and references therein.
	\begin{rmk}
		The homogeneous isotropic case
		 \begin{equation}\label{isotropic}
			\C_i(x)A:=\lambda_i(\tr A )I+2\mu_i A_{\rm sym},
		\end{equation}
	with the Lamé constants satisfying $\mu_i>0$ and $n\lambda_i+2\mu_i>0$, fulfils the previous assumptions \ref{hyp:C1}-\ref{hyp:C5}. The first four conditions are a direct consequence of the explicit form \eqref{isotropic}; to check the validity of \ref{hyp:C5} we notice that
	\begin{equation*}
		\C_i(x)A:A= \lambda_i (\tr A)^2+2\mu_i |A_{\rm sym}|^2.
	\end{equation*} 
	If $\lambda_i\ge 0$ we conclude by choosing $c_i=2\mu_i$, otherwise by using the inequality $(\tr A)^2\le n|A_{\rm sym}|^2$ we get $\C_i(x)A:A\ge (n\lambda_i+2\mu_i)|A_{\rm sym}|^2$, and so one can take $c_i=n\lambda_i+2\mu_i$.
	\end{rmk}

	\subsection{Cohesive interfacial energy}\label{subsec:cohesive}
	
	The behaviour of the interface between the two layers is assumed to follow a cohesive law with respect to their reciprocal slip. We allow for different loading and unloading regimes, which can be modelled by means of the energy $\mc K\colon L^0(\Omega)^+\times L^0(\Omega)^+\to [0,+\infty)$ defined by
	\begin{equation}\label{cohesiveenergy}
		\mc K(\delta,\gamma):=\int_{\Omega}\Phi(\delta(x),\gamma(x))\d x,
	\end{equation}
	for a suitable cohesive energy density $\Phi$ described below.
	
	If $t\mapsto \bm{u}(t)$ represents the evolution of the two displacements, the first variable $\delta=\delta(t)$ in \eqref{cohesiveenergy} plays the role of the size of the actual slip between the two layers, namely $\delta(t)=|u_1(t)-u_2(t)|$, while the second one $\gamma=\gamma(t)$, which takes into account irreversible effects in the interface, describes the \lq\lq maximal" amount of slip reached during the evolution till a certain time $t$ (see \eqref{historyslip} for the rigorous definition).
	
	During the (dissipative) loading phase, namely when $\delta(t)=\gamma(t)$, the cohesive behaviour is described by the concavity property of the function $z\mapsto\Phi(z,z)$, coherently with Barenblatt's theory \cite{Barenblatt}. On the other hand, in the unloading regime $\delta(t)<\gamma(t)$ the overall behaviour is elastic and thus $y\mapsto \Phi(y,\gamma)$ shall be quadratic. 
	
	\begin{figure}
		\includegraphics[scale=.9]{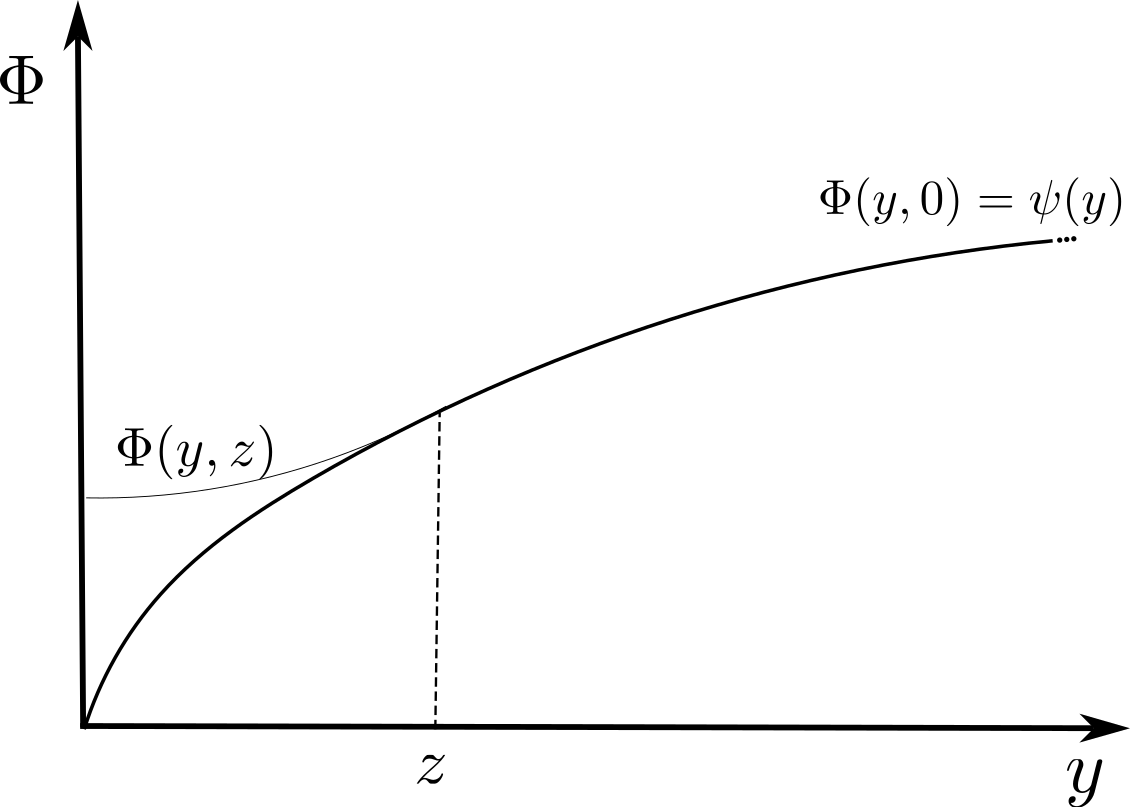}\caption{Graph of the cohesive energy density $\Phi$.}\label{fig:1}
	\end{figure}
		
	In order to incorporate these features (see also \cite{NegSca,NegVit} and Figure~\ref{fig:1}), in this paper we consider a cohesive energy density $\Phi\colon [0,+\infty)^2\to [0,+\infty)$ of the form
	\begin{equation}\label{Phi}
		\Phi(y,z):=\begin{cases}\displaystyle
			\frac{\psi'(z)}{2z}y^2+\psi(z)-\frac{z\psi'(z)}{2}, &\text{if }y<z,\\
			\psi(y), &\text{otherwise},
		\end{cases}
	\end{equation}
	where the function $\psi\colon [0,+\infty)\to [0,+\infty)$, which governs the loading regime, is assumed to be strictly increasing, bounded, concave, of class $C^2$ and such that $\psi'(0)>0=\psi(0)$ and $\psi''(z)\ge -\lambda$ for all $z\ge 0$ and for some $\lambda>0$. This last condition is equivalent to the so-called $\lambda$-convexity, namely 
	\begin{equation}\label{lambdaconv}
		\psi(\theta z^a+(1-\theta)z^b)\le \theta\psi(z^a)+(1-\theta)\psi(z^b)+\frac\lambda 2\theta(1-\theta)|z^a-z^b|^2,\quad\text{for all }\theta\in [0,1],\,z^a,z^b\in [0,+\infty).
	\end{equation}	
	The simplest example of function $\psi$ fulfilling the previous assumptions is given by the negative exponential
	\begin{equation*}
		\psi(z)=\kappa(1-e^{-\rho z}), \qquad\text{for $\kappa, \rho>0$}.	
	\end{equation*}	
\begin{rmk}
	The analysis contained in the present paper can be extended, with minor changes, to the case of a function $\psi$ which is definitively constant, modelling the occurence of complete delamination in the interface. For instance, we can also consider
	\begin{equation*}
		\psi(z)=\begin{cases}\displaystyle
					\kappa\frac z\delta\left(\frac {z^2}{\delta^2}-3\frac z\delta+3\right),&\text{if }z\in[0,\delta],\\
					\kappa,&\text{if }z>\delta,
				\end{cases} \qquad\text{for $\kappa, \delta>0$}.	
\end{equation*}	
We refer to \cite[Lemma~2.4 and equation (2.12)]{BonCavFredRiva} for more details.
\end{rmk}

	\begin{rmk}
		In the one-dimensional case studied in \cite{BonCavFredRiva}, the cohesive energy density is represented by a function $\varphi$ defined on the set $\{z\ge y\ge 0\}$, which slightly differs (see (2.11) therein) from the function $\Phi$ here considered and defined in \eqref{Phi}. Actually, the two formulations are completely equivalent, indeed one can easily check that $\Phi(y,z)=\varphi(y,z\vee y)$ for all $(y,z)\in [0,+\infty)^2$. However, working with $\Phi$ instead of $\varphi$ makes several computations lighter; this fact motivates our choice.
	\end{rmk}

	The following proposition collects the main properties of the density $\Phi$.
	\begin{prop}\label{propPhi}
		The function $\Phi$ defined in \eqref{Phi} fulfils:
		\begin{itemize}
			\item[(i)] $\Phi$ is nonnegative, bounded and continuous on the whole $[0,+\infty)^2$;
			\item[(ii)] for all $z\ge 0$ the function $\Phi(\cdot,z)$ is nondecreasing, Lipschitz and of class $C^1$ in $[0,+\infty)$. Moreover there holds $0\le \partial_y\Phi(y,z)\le \partial_y\Phi(z,z)= \psi'(z)\le \psi'(0)$ for all $y,z\ge 0$ and, if $z>0$, we also have $\partial_y\Phi(0,z)=0$. Furthermore $\partial_y\Phi$ is continuous in $[0,+\infty)^2\setminus\{(0,0)\}$;
			\item[(iii)] for all $y\ge 0$ the function $\Phi(y,\cdot)$ is nondecreasing and of class $C^1$ in $[0,+\infty)$. Moreover it is strictly increasing in $[y,+\infty)$. Furthermore $\partial_z\Phi$ is continuous and positive on the set $\{z>y\ge 0\}$;
			\item[(iv)] for all $z\ge 0$ the function $\Phi(\cdot,z)$ is $\lambda$-convex.
			\item[(v)] 	$\Phi(y,z)=\Phi(y,z\vee y)$ for all $(y,z)\in [0,+\infty)^2$.
		\end{itemize}
	\end{prop}
	 \begin{proof}
	 	The continuity of $\Phi$ easily follows from the explicit form \eqref{Phi} recalling that $\psi'$ is continuous. Observing that $\Phi(\cdot,z)$ is nondecreasing for fixed $z\ge 0$ (it consists of a parabola followed by the nondecreasing function $\psi$) one has
	 	\begin{equation*}
	 		\Phi(0,z)\le\Phi(y,z)\le \lim\limits_{s\to +\infty}\Phi(s,z)=\sup\limits_{s\ge 0} \psi(s),
	 	\end{equation*}
 	and so $\Phi$ is bounded. Since $\psi$ is nondecreasing, concave and smooth, it is straightforward to check that the function $\Phi(0,z)=\psi(z)-\frac{z\psi'(z)}{2}$ is nondecreasing, from which one deduces
 	\begin{equation*}
 		\Phi(y,z)\ge\Phi(0,z)\ge\Phi(0,0)=\psi(0)=0.
 	\end{equation*}
 Hence $\Phi$ is nonnegative and $(i)$ is proved.
  
 We now focus on $(ii)$. If $z=0$, the statement is true since $\Phi(y,0)=\psi(y)$. If $z>0$, one has
 \begin{equation*}
 	\partial_y\Phi(y,z)=\begin{cases}
 		\frac{\psi'(z)}{z}y,&\text{if }y<z,\\
 		\psi'(y),&\text{if }y\ge z,
 	\end{cases}
 \end{equation*}
and the validity of $(ii)$ can be inferred from the above explicit formula.

To check $(iii)$ it is enough to notice that there holds
\begin{equation*}
	\partial_z \Phi(y,z)=\begin{cases}
		0,&\text{if }z<y,\\
		\frac{\psi'(z)-z\psi''(z)}{2}\left(1-\frac{y^2}{z^2}\right),&\text{if }z\ge y.
	\end{cases}
\end{equation*}
Indeed, from the assumptions on $\psi$, we can deduce that $\partial_z \Phi(y,\cdot)$ is nonnegative and continuous in the whole $[0,+\infty)$ and positive in $(y,+\infty)$. Analogously, one can prove that $\partial_z\Phi$ is continuous and positive if $z>y$.

Property $(iv)$ is an immediate consequence of the $\lambda$-convexity of $\psi$. Indeed, $\Phi(\cdot,z)$ is composed by a convex function (a parabola) in $[0,z]$ and by $\psi$ in $[z,+\infty)$.

Finally, property $(v)$ follows from the very definition \eqref{Phi}. 
	 \end{proof}
	\subsection{External loading and initial conditions}\label{subsec:loading}
	The evolution of the system is driven by an external horizontal loading $w$ acting on a portion of the boundary $\partial_D\Omega\subseteq\partial\Omega$ with positive Hausdorff measure, i.e. $\mc H^{n-1}(\partial_D\Omega)>0$.  We restrict ourselves to \lq\lq slow\rq\rq loadings, so that inertial effects may be neglected and the resulting evolution turns out to be quasistatic.
	
	As usual in the mathematical treatment of mechanical models, the external loading is assumed to be the trace of a function defined on the whole of $\Omega$. In this paper we require 
	\begin{equation}\label{externalloading}
		w\in AC([0,T]; H^1(\Omega;\R^n)),
	\end{equation}
	where $T>0$ is an arbitrary time horizon. For the sake of brevity, given a function $f\colon\partial_D\Omega\to \R^n$, we introduce the following notation:
	\begin{equation*}
		H^1_{D,f}:=\{v\in H^1(\Omega;\R^n):\, v=f\quad\mc H^{n-1}\text{-a.e. in }\partial_D\Omega\}.
	\end{equation*}

	At the initial time $t=0$ the configuration of the body is described by the initial displacement
	\begin{subequations}\label{initial}
			\begin{equation}\label{initialdisplacement}
			\bm u^0\in (H^1_{D,w(0)})^2.
		\end{equation}
		For technical reasons (see Proposition~\ref{unifgamma}) we will also need to require
		\begin{equation}\label{Lipinitial}
			\bm u^0\in C^{0,1}_{\rm loc}(\Omega;\R^n)^2.
		\end{equation}
	\end{subequations}

\subsection{Energetic solutions}
	The total energy of the system is thus described by the functional $\mc F\colon [0,T]\times H^1(\Omega;\R^n)^2\times L^0(\Omega)^+ \to [0,+\infty]$ given by
	\begin{equation}\label{totalenergy}
		\mc F(t,\bm u, \gamma):=\begin{cases}
			\mc E(\bm u)+\mc K(|u_1-u_2|,\gamma),&\text{if }\bm u\in (H^1_{D,w(t)})^2,\\
			+\infty,&\text{otherwise.}
		\end{cases}
	\end{equation}
	
In order to ensure some convexity of $\mc F$ (see Lemma~\ref{lemma:convexity}), we will require that
\begin{equation}\label{eq:hyplambda}
	\lambda<\frac{c_1\wedge c_2}{2 K_2^2},
\end{equation}
where $\lambda$ is the constant appearing in \eqref{lambdaconv}, $c_1,c_2$ are given by \ref{hyp:C5}, while $K_2$ is the Korn's constant from Proposition~\ref{prop:korn} for $p=2$.

Before presenting the definition of solution for the model under consideration we introduce the following notation. Given an arbitrary family $\{f_j\}_{j\in J}\subseteq L^0(\Omega)^+$, the essential (or lattice) supremum 
\begin{equation*}
	f=\essup_{j\in J}f_j,
\end{equation*}
of the family is defined as the unique function in $L^0(\Omega)^+$ satisfying the two properties:
\begin{itemize}
	\item for every $j\in J$ one has $f\ge f_j $ a.e. in $\Omega$;
	\item if $g\in L^0(\Omega)^+$ and for every $j\in J$ there holds $g\ge f_j $ a.e. in $\Omega$, then $g\ge f$ a.e. in $\Omega$. 
\end{itemize}
It is well-known that the essential supremum $f$ always exists; moreover, see for instance \cite[Lemma~2.6.1]{MeyNie}, such $f$ can be computed as a pointwise supremum over a countable subset $J^\N$ of $J$, namely
\begin{equation}\label{eq:pointsup}
	f(x)=\sup\limits_{j\in J^\N}f_j(x),\quad\text{for a.e. }x\in\Omega.
\end{equation}
\begin{defi}\label{def:enev}
	Given an external loading $w$ and an initial condition $\bm u^0$ satisfing \eqref{externalloading} and \eqref{initialdisplacement}, we say that a function $\bm u\in B([0,T]; H^1(\Omega;\R^n)^2)$ is an \emph{energetic solution} of the cohesive interface model if the initial condition $\bm u(0)=\bm u^0$ is attained and if the following global stability condition and energy balance are satisfied for all $t\in [0,T]$:
	\begin{enumerate}[label=\textup{(GS)}]
		\item \label{GS} $\mc F(t,\bm u(t),\delta_h(t))\le \mc F(t,\bm v,\delta_h(t)),\quad \text{ for every }\bm v\in H^1(\Omega;\R^n)^2;$
	\end{enumerate}
	\begin{enumerate}[label=\textup{(EB)}]
		\item \label{EB} $\displaystyle \mc F(t,\bm u(t),\delta_h(t))=\mc F(0,\bm u^0,|u_1^0-u_2^0|)+\mc W(t);$
	\end{enumerate}
where $\delta_h(t)=\delta_h[\bm u](t)$ is the history slip defined by 
\begin{equation}\label{historyslip}
	\delta_h(t):= \essup_{s\in [0,t]}|u_1(s)-u_2(s)|,
\end{equation}
while $\mc W(t)$ represents the work of the external forces and has the form
\begin{equation}\label{work}
	\mc W(t)=\int_{0}^{t}\int_{\Omega}\sum_{i=1}^2\C_i e(u_i(s)):e(\dot{w}(s))\d x\d s.
\end{equation}
\end{defi}
\begin{rmk}
	The choice of working with energetic solutions is motivated by the convexity of the energy $\mc F(t,\cdot,\gamma)$, see Lemma~\ref{lemma:convexity}. We refer to the monograph \cite{MielkRoubbook} for an exhaustive survey on the various notions of solution in quasistatic regimes.
\end{rmk}

	By condition \ref{GS}, for the existence of an energetic solution it is necessary that the initial datum $\bm u^0$ fulfils
\begin{equation}\label{Mininitial}
	\mc E(\bm u^0)+\mc K(|u^0_1-u^0_2|,|u^0_1-u^0_2|)\le \mc E(\bm v)+\mc K(|v_1-v_2|,|u^0_1-u^0_2|),\quad \text{ for every }\bm v\in (H^1_{D,w(0)})^2.
\end{equation}
Notice that it is not clear whether a function $\bm u^0$ satisfying \eqref{Mininitial} exists in general, neither whether \eqref{Mininitial} is compatible with \eqref{Lipinitial}. However, if $w(0)=0$ (i.e. the external loading is initially null, which is a reasonable assumption in view of mechanical applications) the choice $\bm u^0=\bm 0$ complies with both conditions.

Our main result, regarding existence and certain regularity properties of energetic solutions, is stated in the following theorem, whose proof will be the content of Sections~\ref{sec:regularizedenergy} and \ref{sec:proof}.
\begin{thm}\label{mainthm}
	Assume \eqref{Omega}, assume \ref{hyp:C1}-\ref{hyp:C5}, and let the cohesive energy density $\Phi$ be of the form \eqref{Phi}, where the function $\psi$ is as in Section~\ref{subsec:cohesive} and \eqref{eq:hyplambda} is in force. Then, given an external loading $w$ and an initial condition $\bm u^0$ satisfing \eqref{externalloading}, \eqref{initial} and \eqref{Mininitial}, there exists an energetic solution $\bm u$ of the cohesive interface model in the sense of Definition~\ref{def:enev}.
	
	Moreover, such function $\bm u$ actually belongs to $AC([0,T]; H^1(\Omega;\R^n)^2)\cap B([0,T];C^{0,\alpha}(\overline{\Omega'};\R^n)^2)$ for all $\Omega'\subset\subset \Omega$ and $\alpha\in (0,1)$, and so there also holds $\bm u\in C^0([0,T]\times\Omega)$. In particular the history slip $\delta_h$, defined in \eqref{historyslip}, can be computed as a pointwise supremum and belongs to $C^0([0,T]\times\Omega)$ as well.
\end{thm}

\section{Regularized energy}\label{sec:regularizedenergy}

As mentioned in the Introduction, a first step towards the proof of Theorem~\ref{mainthm} consists in the regularization of the cohesive energy \eqref{cohesiveenergy}, in order to get rid of the kink of $\Phi$ at the origin. This procedure will allow us to compute the Euler-Lagrange equations of the regularized version of the total energy \eqref{totalenergy} (see Proposition~\ref{propeleq}), and to apply elliptic regularity theory, which in turn will be a key ingredient in order to manage the history slip.

To this aim we first approximate the function $\psi$. For $\eps>0$ we define
\begin{equation*}
	\psi_\eps(z):=\begin{cases}\displaystyle
		\frac{z^2}{2\eps}, &\text{if }z\in[0,z_\eps],\\\displaystyle
		\psi(z)-\psi(z_\eps)+\frac{z_\eps^2}{2\eps},&\text{if }z>z_\eps,
	\end{cases}
\end{equation*}
where $z_\eps$ is the unique fixed point of the map $s\mapsto \eps\psi'(s)$. Notice that $\psi_\eps$ has been constructed in such a way that
\begin{equation}\label{apprder}
	\psi'_\eps(z)=\frac{z}{\eps}\wedge \psi'(z),\quad\text{for }z\ge 0,
\end{equation}
see also Figure~\ref{fig:2}. The same regularization has been used in \cite[Section 4]{NegVit}, with different scopes. 

\begin{figure}
	\subfloat{\includegraphics[scale=.78]{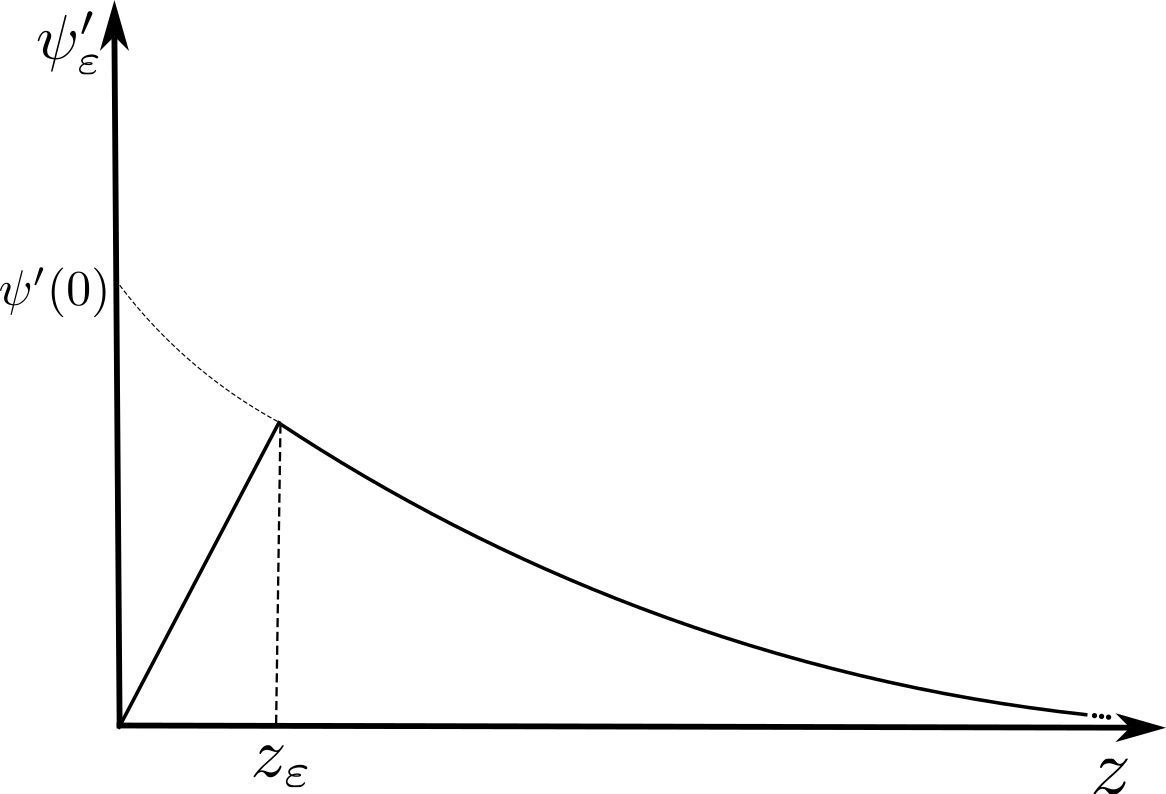}}\quad\subfloat{\includegraphics[scale=.78]{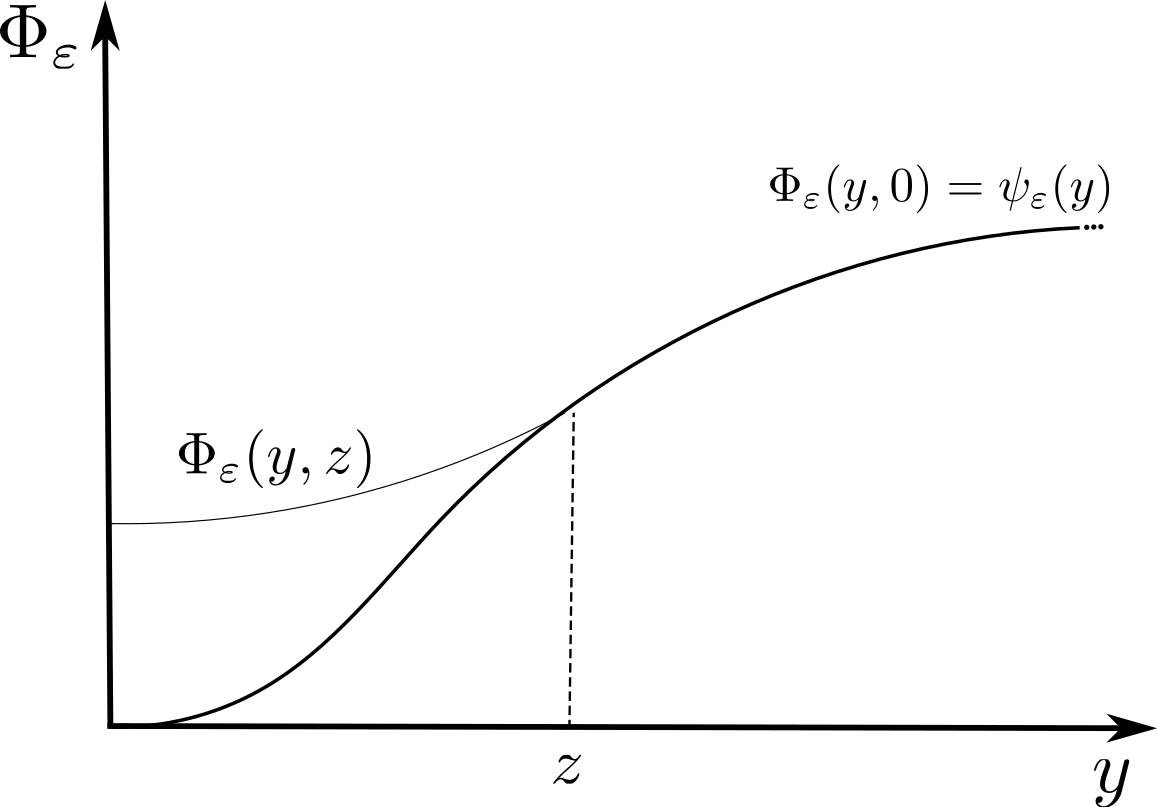}}\caption{Graph of $\psi'_\eps$ and of the regularized cohesive density $\Phi_\eps$.}\label{fig:2}
\end{figure}

By simple computations one can prove that $\psi_\eps$ is nondecreasing, bounded and Lipschitz (uniformly with respect to $\eps$), of class $C^1$ on $[0,+\infty)$ and satisfies $\psi_\eps'(0)=0=\psi_\eps(0)$. Furthermore, observe that $\lim\limits_{\eps\to 0}z_\eps=0$ and $\lim\limits_{\eps\to 0}z_\eps/\eps=\psi'(0)$, whence also $z_\eps^2/\eps$ vanishes as $\eps\to 0$. From this facts, one can easily show that
\begin{equation}\label{unifconv}
	\psi_\eps\xrightarrow[\eps\to 0]{}\psi,\quad\text{uniformly in }[0,+\infty).
\end{equation}

 The regularized cohesive density $\Phi_\eps$ is then defined by \eqref{Phi} replacing $\psi$ with $\psi_\eps$, and analogously we obtain the regularized cohesive energy $\mc K_\eps$ and the regularized total energy $\mc F_\eps$.

The following proposition collects the main properties of such approximation $\Phi_\eps$.
\begin{prop}\label{propphieps}
	The function $\Phi_\eps$ satisfies:
	\begin{itemize}
		\item[(j)] $\Phi_\eps$ is nonnegative, continuous on the whole $[0,+\infty)^2$, and bounded uniformly in $\eps$;
		\item[(jj)] for all $z\ge 0$ the function $\Phi_\eps(\cdot,z)$ is nondecreasing, Lipschitz and of class $C^1$ in $[0,+\infty)$. Moreover there holds $0\le \partial_y\Phi_\eps(y,z)\le \psi'(0)$ for all $y,z\ge 0$ and $\partial_y\Phi_\eps(0,z)=0$;
		\item[(jjj)] $\Phi_\eps\xrightarrow[\eps\to 0]{}\Phi$ uniformly in $[0,+\infty)^2$.
	\end{itemize}
\end{prop}
\begin{proof}
	The proof of $(j)$ and $(jj)$ is similar to the one of Proposition~\ref{propPhi}, so we do not write the details. 
	
	To prove $(jjj)$ we fix $(y,z)\in [0,+\infty)^2$ and we estimate
	\begin{equation}\label{eq1}
		|\Phi_\eps(y,z)-\Phi(y,z)|\le\int_{0}^{y}|\partial_y\Phi_\eps(s,z)-\partial_y\Phi(s,z)|\d s+|\Phi_\eps(0,z)-\Phi(0,z)|.
	\end{equation}
	By observing that by the expression \eqref{apprder} for any $x\ge 0$ there holds
	\begin{equation*}
		|\psi'_\eps(x)-\psi'(x)|\le \psi'(0)\chi_{[0,z_\eps]}(x),
	\end{equation*}
we first deduce that
\begin{equation}\label{eq2}
	\begin{aligned}
	|\Phi_\eps(0,z)-\Phi(0,z)|&=\left|\psi_\eps(z)-\frac{z\psi'_\eps(z)}{2}-\psi(z)+\frac{z\psi'(z)}{2}\right|\le |\psi_\eps(z)-\psi(z)|+\frac{z}{2}|\psi'_\eps(z)-\psi'(z)|\\
	&\le \sup\limits_{s\ge 0}|\psi_\eps(s)-\psi(s)|+\frac{\psi'(0)}{2}z\chi_{[0,z_\eps]}(z)\le \sup\limits_{s\ge 0}|\psi_\eps(s)-\psi(s)|+\frac{\psi'(0)}{2}z_\eps.
\end{aligned}
\end{equation}

We now claim that for all $s,z\ge 0$ there holds
\begin{equation}\label{claim}
	|\partial_y\Phi_\eps(s,z)-\partial_y\Phi(s,z)|\le \psi'(0)\chi_{[0,z_\eps]}(s\vee z).
\end{equation}
If the claim is true we obtain
\begin{equation}\label{eq3}
	\int_{0}^{y}|\partial_y\Phi_\eps(s,z)-\partial_y\Phi(s,z)|\d s\le \psi'(0)	\int_{0}^{y} \chi_{[0,z_\eps]}(s\vee z)\d s\le 2\psi'(0)z_\eps,
\end{equation}
and we conclude by combining \eqref{eq1}, \eqref{eq2}, \eqref{eq3} and recalling \eqref{unifconv}.

We are only left to show \eqref{claim}: if $z=0$ we have
\begin{equation*}
	|\partial_y\Phi_\eps(s,0)-\partial_y\Phi(s,0)|=|\psi'_\eps(s)-\psi'(s)|\le \psi'(0)\chi_{[0,z_\eps]}(s).
\end{equation*}
If $z>0$ we first consider the case $s<z$: so we have
\begin{equation*}
	|\partial_y\Phi_\eps(s,z)-\partial_y\Phi(s,z)|=\frac sz|\psi'_\eps(z)-\psi'(z)|\le \psi'(0)\chi_{[0,z_\eps]}(z).
\end{equation*}
If instead $s\ge z$ there holds
\begin{equation*}
	|\partial_y\Phi_\eps(s,z)-\partial_y\Phi(s,z)|=|\psi'_\eps(s)-\psi'(s)|\le \psi'(0)\chi_{[0,z_\eps]}(s).
\end{equation*}
Combining the three above cases we conclude.
\end{proof}

\subsection{Time-discretization scheme}\label{subsec:minmov}

	We now employ the classical Minimizing Movements argument with the regularized energy $\mc F_\eps$. Let $\tau>0$ be a small parameter such that $T/\tau\in \N$, and for $k=0,\dots, T/\tau$ let $t^k:=k\tau$, so that the family $\{t^k\}_{k=0,\dots, T/\tau}$ forms an equidistant partition of $[0,T]$. For $k=1,\dots,T/\tau$, we consider the following recursive algorithm: given the previous pair $(\bm u^{k-1}_\eps,\gamma^{k-1}_\eps)$, we set
	\begin{equation}\label{discralg}
		\begin{cases}\displaystyle
			\bm u^k_\eps\in\argmin\limits_{\bm{v}\in H^1(\Omega;\R^n)^2}\mc F_\eps(t^k,\bm{v}, \gamma^{k-1}_\eps),\\
			\gamma^k_\eps:=\gamma^{k-1}_\eps\vee|(u_1)^k_\eps-(u_2)^k_\eps|,			
		\end{cases}
	\end{equation}
where the initial conditions are given by
\begin{equation*}
	\bm u^0_\eps:=\bm u^0,\quad \gamma^0_\eps:=|u^0_1-u^0_2|.
\end{equation*}

	We observe that the minimization is well posed since $\mc F_\eps (t^k,\cdot, \gamma^{k-1}_\eps)$ is coercive and lower semicontinuous in the weak topology of $H^1(\Omega;\R^n)^2$: coercivity follows by \ref{hyp:C5} together with Korn-Poincar\'e inequality \eqref{kornpoincare}, while semicontinuity is standard for $\mc E$ and is a consequence of Fatou's Lemma for $\mc K_\eps(\cdot,\gamma^{k-1}_\eps)$.

	We introduce the following notation: given a vector $v\in \R^n$ we denote by $\dir v\in \R^n$ its direction, namely 
	\begin{equation*}
		\dir v:=\begin{cases}
			\frac{v}{|v|},&\text{if }v\neq 0,\\
			0,&\text{if }v= 0.
		\end{cases}
	\end{equation*} 
	\begin{prop}\label{propeleq}
		For all $k=1,\dots,T/\tau$ the function $\bm u^k_\eps\in (H^1_{D,w(t^k)})^2$ is a weak solution of the system
		\begin{equation}\label{eleq}
			\begin{cases}
				-\div(\mathbb{C}_1 e(u_1))=-\partial_y\Phi_\eps(|u_1-u_2|,\gamma^{k-1}_\eps)\dir(u_1-u_2),&\text{in }\Omega,\\
				-\div(\mathbb{C}_2 e(u_2))=\partial_y\Phi_\eps(|u_1-u_2|,\gamma^{k-1}_\eps)\dir(u_1-u_2),&\text{in }\Omega,\\
				\mathbb{C}_1 e(u_1)n_\Omega=\mathbb{C}_2 e(u_2)n_\Omega=0,&\text{in }\partial\Omega\setminus\partial_D\Omega,
			\end{cases}
		\end{equation}
	where $n_\Omega$ denotes the outward unit normal to the set $\Omega$.
	\end{prop}
\begin{proof}
	It is enough to show that system \eqref{eleq} is the Euler-Lagrange equation of $\mc F_\eps$. By using standard variations $\bm u^k_\eps+h\bm\varphi$, with $\bm\varphi\in (H^1_{D,0})^2$, from the minimality of $\bm u^k_\eps$ one deduces
	\begin{align*}
		0\le&\quad\,\sum_{i=1}^2\int_{\Omega}\mathbb C_ie((u_i)^k_\eps):e(\varphi_i)\d x\\
		&+\lim\limits_{h\to 0^+}\int_{\Omega}\frac{\Phi_\eps(|(u_1)^k_\eps-(u_2)^k_\eps+h(\varphi_1-\varphi_2)|,\gamma^{k-1}_\eps)-\Phi_\eps(|(u_1)^k_\eps-(u_2)^k_\eps|,\gamma^{k-1}_\eps)}{h}\d x.
	\end{align*}
	Since $\Phi_\eps(\cdot,z)$ is Lipschitz and smooth (see $(jj)$ in Proposition~\ref{propphieps}) we can move the limit inside the integral, and by taking also $-\bm\varphi$ as a test function we finally get
	\begin{align*}
		0&=\sum_{i=1}^2\int_{\Omega}\mathbb C_ie((u_i)^k_\eps):e(\varphi_i)\d x+\int_{\Omega}\partial_y\Phi_\eps(|(u_1)^k_\eps{-}(u_2)^k_\eps|,\gamma^{k-1}_\eps)\dir((u_1)^k_\eps{-}(u_2)^k_\eps)\cdot(\varphi_1-\varphi_2)\d x,
	\end{align*}
and so we conclude.
\end{proof}

\subsection{Uniform bounds}

In this section we deduce uniform estimates for the pairs $(\bm u^k_\eps, \gamma^k_\eps)$ previously obtained. The bounds in the Sobolev space $H^1(\Omega;\R^n)$ stated in Proposition~\ref{prop:boundH1} directly follow from energetic arguments, while the ones in H\"older spaces stated in Proposition~\ref{unifgamma} are less standard and are a consequence of elliptic regularity.

\begin{prop}\label{prop:boundH1}
	Assume \eqref{externalloading} and \eqref{initialdisplacement}. Then there exists a constant $C>0$ independent of $\eps$ and $\tau$ such that
	\begin{equation}\label{boundH1}
		\max\limits_{k=0,\dots,T/\tau}\Vert \bm u^k_\eps\Vert_{H^{1}(\Omega)^2}\le C.
	\end{equation}
\end{prop}
\begin{proof}
	We fix $k=1,\dots T/\tau$ and in \eqref{discralg}$_1$ we use the function $\bm w(t^k):=(w(t^k),w(t^k))$ as a competitor. By exploiting \ref{hyp:C1}, \ref{hyp:C5} together with the nonnegativity and boundedness (uniformly in $\eps$) of $\Phi_\eps$ we obtain
	\begin{align*}
		\sum_{i=1}^2 \frac{c_i}{2}\Vert e((u_i)^k_\eps)\Vert^2_{L^2(\Omega)}&\le \mc F_\eps(t^k,\bm u^k_\eps,\gamma^{k-1}_\eps)\le \mc F_\eps(t^k,\bm w(t^k),\gamma^{k-1}_\eps)\\
		&\le C\Vert w(t^k)\Vert^2_{H^{1}(\Omega)}+\int_{\Omega}\Phi_\eps(0,\gamma^{k-1}_\eps)\d x\\
		&\le C(\max\limits_{t\in[0,T]}\Vert w(t)\Vert^2_{H^{1}(\Omega)}+1).
	\end{align*}
By means of Korn-Poincar\'e inequality \eqref{kornpoincare} and exploiting \eqref{externalloading} we now deduce
\begin{equation*}
	\max\limits_{k=1,\dots, T/\tau}\Vert \bm u^k_\eps\Vert_{H^{1}(\Omega)^2}\le C,
\end{equation*}
and so we conclude since for $k=0$ one has $\bm u^0_\eps=\bm u^0$.
\end{proof}

We will take advantage of the following regularity result. We refer to \cite[Theorem~7.2]{GiacMart} for its proof.
\begin{thm}\label{thmregularity}
	Let $\Omega\subseteq \R^n$ satisfy \eqref{Omega}. Let $u\in H^1(\Omega;\R^n)$ be a weak solution of the equation
	\begin{equation*}
		-\div(\mathbb{C}e(u))=g,\qquad\text{in }\Omega,
	\end{equation*}
where the tensor $\mathbb{C}$ satisfies \ref{hyp:C1}, \ref{hyp:C3} and \eqref{LegendreHadamard}.

If $g\in L^{\frac{np}{n+p}}(\Omega;\R^n)$ for some $p>2$, then $\nabla u\in L^p_{\rm loc}(\Omega;\R^{n\times n})$ and for all open set $\Omega'\subset\subset\Omega$ there holds
\begin{equation*}
	\Vert\nabla u\Vert_{L^p(\Omega')}\le C\Big(\Vert g\Vert_{L^{\frac{np}{n+p}}(\Omega)}+\Vert\nabla u\Vert_{L^2(\Omega)}\Big),
\end{equation*}
where the constant $C>0$ depends only on $p,n,c,\omega$ and $\mathrm{dist}(\Omega',\Omega)$.
\end{thm}

\begin{prop}\label{unifgamma}
		Assume \eqref{Omega}, \eqref{externalloading} and \eqref{initialdisplacement}. Then for all $k=1,\dots, T/\tau$ there actually holds\begin{equation*}
		\bm u^k_\eps\in W^{1,p}_{\rm loc}(\Omega;\R^n)^2\quad \text{for all }p>2,
	\end{equation*}
	and for all $\Omega'\subset\subset\Omega$ there exists a constant $C>0$ independent of $\eps$ and $\tau$ (possibly depending on $p$ and $\Omega'$) such that
	\begin{equation}\label{boundW1p}
	\max\limits_{k=1,\dots, T/\tau}\Vert \bm u^k_\eps\Vert_{W^{1,p}(\Omega')^2}\le C.
	\end{equation}
Furthermore, assuming \eqref{Lipinitial}, one also has that for all $k=0,\dots, T/\tau$ the function $\gamma^k_\eps$ is locally $\alpha$-H\"older continuous for every $\alpha\in (0,1)$, and for all $\Omega'\subset\subset\Omega$ there exists $C$ as before such that
\begin{equation}\label{boundlipgamma}
	\max\limits_{k=0,\dots, T/\tau}\Vert  \gamma^k_\eps\Vert_{\rm{C^{0,\alpha}}(\overline\Omega')}\le C.
\end{equation}
\end{prop}
\begin{proof}
	Let us fix $k=1,\dots, T/\tau$ and $p>2$. By Proposition~\ref{propeleq} we know that $\bm u^k_\eps$ is a weak solution of
	\begin{equation*}
		-\div (\mathbb C_i e(u_i))=g_i,
	\end{equation*}
	where for $i=1,2$ we set $g_i=(-1)^i \partial_y\Phi_\eps(|(u_1)^k_\eps-(u_2)^k_\eps|,\gamma^{k-1}_\eps)\dir((u_1)^k_\eps-(u_2)^k_\eps)$. Observe that by $(jj)$ in Proposition~\ref{propphieps} we have $g_i\in L^\infty(\Omega;\R^n)$ with
	\begin{equation}\label{boundgi}
		\Vert g_i\Vert_{L^\infty(\Omega)}\le \psi'(0).
	\end{equation}
By Theorem~\ref{thmregularity} we hence deduce that $\nabla(u_i)^k_\eps\in L^q_{\rm loc}(\Omega;\R^{n\times n})$ for all $q>2$ with
\begin{equation}\label{boundgrad}
\Vert\nabla (u_i)^k_\eps\Vert_{L^q(\Omega')}\le C\Big(\Vert g_i\Vert_{L^{\frac{nq}{n+q}}(\Omega)}+\Vert\nabla (u_i)^k_\eps\Vert_{L^2(\Omega)}\Big)\le C(\psi'(0)+1),
\end{equation}
for all $\Omega'\subset\subset\Omega$ (the constant $C$ may depend on $\Omega'$). In the last inequality above we used \eqref{boundgi} together with \eqref{boundH1}.

In particular, we can choose $q=p$ in \eqref{boundgrad}: we are thus left to control the $L^p$-norm of $(u_i)^k_\eps$ in order to prove \eqref{boundW1p}. This easily follows by means of a bootstrap argument which combines \eqref{boundH1}, Sobolev Embedding Theorem~\ref{SobEmb}, and \eqref{boundgrad}. 

Let us now prove \eqref{boundlipgamma}, assuming in addition \eqref{Lipinitial}. Fix $k=0,\dots, T/\tau$ and $\alpha\in (0,1)$. If $k=0$ there is nothing to prove, since $\gamma^0_\eps=|u^0_1-u^0_2|\in C^{0,1}_{\rm loc}(\Omega)\subseteq C^{0,\alpha}_{\rm loc}(\Omega)$. If $k\ge 1$, let $p>\frac{n}{1-\alpha}$ so that $1-\frac np>\alpha$. By means of \eqref{boundW1p}, Sobolev Embedding Theorem yields that $\bm u^j_\eps\in C^{0,\alpha}_{\rm loc}(\Omega;\R^n)^2$ for all $j=1,\dots, k$, with
\begin{equation}\label{holdbound}
	\max\limits_{j=1,\dots, k}\Vert \bm u^j_\eps \Vert_{C^{0,\alpha}(\overline\Omega')^2}\le C \max\limits_{j=1,\dots, k}\Vert \bm u^j_\eps\Vert_{W^{1,p}(\Omega')^2}\le C,
\end{equation}
for all $\Omega'\subset\subset\Omega$ (the constant $C$ may depend on $\Omega'$). In particular $\gamma^k_\eps$ is continuous, since it is a maximum among a finite number of  continuous functions, and for a fixed $\Omega'\subset\subset\Omega$ there holds
\begin{align*}
	\Vert\gamma^k_\eps\Vert_{C^{0}(\overline{\Omega}')}=\max\limits_{j=1,\dots, k}\Vert (u_1)^j_\eps-(u_2)^j_\eps\Vert_{C^{0}(\overline{\Omega}')}\vee \Vert u^0_1-u^0_2\Vert_{C^{0}(\overline{\Omega}')}\le C.
\end{align*}
	We finally show that the seminorms $[\gamma^k_\eps]_{\alpha,\overline{\Omega'}}$ are uniformly bounded as well, thus concluding the proof. Let us fix $x,y\in \overline{\Omega'}$, with $x\neq y$; so there exists $\bar j=\bar j(x)\in \{0,\dots k\}$ such that $\gamma^k_\eps(x)=|(u_1)^{\bar{j}}_\eps(x)-(u_2)^{\bar{j}}_\eps(x)|$. By using \eqref{holdbound} we now deduce
	\begin{align*}
		\gamma^k_\eps(x)&\le |(u_1)^{\bar{j}}_\eps(y)-(u_2)^{\bar{j}}_\eps(y)|+\sum_{i=1}^2|(u_i)^{\bar{j}}_\eps(x)-(u_i)^{\bar{j}}_\eps(y)|\\
		&\le \gamma^k_\eps(y)+C|x-y|^\alpha,
	\end{align*}
and we obtain the result by the arbitrariness of $x$ and $y$ in $\overline{\Omega'}$.
\end{proof}

As a corollary, we are able to pass to the limit the discrete functions $\bm u^k_\eps$ and $\gamma^k_\eps$ as $\eps\to 0$, obtaining the existence of minimizing movements for the original functional $\mc F$, which in addition preserve the uniform estimates previously obtained (see \eqref{againbounds}). This fact will be crucial to deal with the irreversible functions $\gamma^k$.
\begin{cor}\label{corconv}
	Assume \eqref{Omega}, \eqref{externalloading} and \eqref{initial}. Then there exists a subsequence $\eps_j\to 0$ and for all $k=0,\dots, T/\tau$ there exist $\bm u^k\in (H^1_{D,w(t^k)}\cap C^{0,\alpha}_{\rm loc}({\Omega};\R^n))^2$ and $\gamma^k\in C^{0,\alpha}_{\rm loc}({\Omega})$ for any $\alpha\in (0,1)$ such that for all $k=0,\dots,T/\tau$ there holds
	\begin{equation}\label{conveps}
	\begin{aligned}
	&	\bm u^k_{\eps_j}\xrightharpoonup[j\to+\infty]{H^1(\Omega;\R^n)^2} \bm u^k,\qquad\text{and}\qquad	\bm u^k_{\eps_j}\xrightarrow[j\to+\infty]{} \bm u^k \text{ locally uniformly in $\Omega$},\\
	& \gamma^k_{\eps_j} \xrightarrow[j\to+\infty]{} \gamma^k \text{ locally uniformly in $\Omega$}.
	\end{aligned}
\end{equation}
Moreover, for $k=1,\dots,T/\tau$ the functions $\bm u^k, \gamma^k$ satisfy
\begin{equation}\label{discralg2}
	\begin{cases}\displaystyle
		\bm u^k\in\argmin\limits_{\bm{v}\in H^1(\Omega;\R^n)^2}\mc F(t^k,\bm{v}, \gamma^{k-1}),\\
		\gamma^k=\gamma^{k-1}\vee|(u_1)^k-(u_2)^k|,\\
		\bm u^0=\bm u^0,\quad \gamma^0=|u^0_1-u^0_2|.
	\end{cases}
\end{equation} 
Furthermore, there exists a constant $C>0$ and for all $\Omega'\subset\subset\Omega$ a constant $C'$, both independent of $\tau$, such that
\begin{equation}\label{againbounds}
\begin{gathered}
	\max\limits_{k=0,\dots,T/\tau}\Vert \bm u^k\Vert_{H^{1}(\Omega)^2}\le C,\\
	\max\limits_{k=0,\dots,T/\tau}(\Vert \bm u^k\Vert_{C^{0,\alpha}(\overline\Omega')^2}+\Vert \gamma^k\Vert_{C^{0,\alpha}(\overline\Omega')})\le C'.
\end{gathered}
\end{equation}
\end{cor}
\begin{proof}
	The existence of convergent subsequences as in \eqref{conveps} is a consequence of \eqref{boundH1}, \eqref{boundlipgamma} and \eqref{holdbound} via a standard application of Banach-Alaoglu and Ascoli-Arzel\'a theorems. The same uniform bounds also yield \eqref{againbounds}.
	
	The only nontrivial part of \eqref{discralg2} is the minimality property of $\bm u^k$. To prove it, fix $\bm{v}\in (H^1_{D,w(t^k)})^2$, so that by \eqref{discralg}$_1$ we have
	\begin{equation}\label{eqmin}
		\mc E(\bm u^k_{\eps_j})+\mc K_{\eps_j}(|(u_1)^k_{\eps_j}-(u_2)^{k}_{\eps_j}|,\gamma^{k-1}_{\eps_j})\le \mc E(\bm{v})+\mc K_{\eps_j}(|v_1-v_2|,\gamma^{k-1}_{\eps_j}).
	\end{equation}
 By weak lower semicontinuity of $\mc E$ we deduce $\mc E(\bm u^k)\le\liminf\limits_{j\to +\infty} \mc E (\bm u^k_{\eps_j})$. By using $(j)$ and $(jjj)$ in Proposition~\ref{propphieps} and $(i)$ in Proposition~\ref{propPhi}, and exploiting \eqref{conveps}, Dominated Convergence Theorem instead yields
 \begin{align*}
 	&\lim\limits_{j\to +\infty}\mc K_{\eps_j}(|(u_1)^k_{\eps_j}-(u_2)^k_{\eps_j}|,\gamma^{k-1}_{\eps_j})=\mc K(|(u_1)^k-(u_2)^k|,\gamma^{k-1}),\\
 	&\lim\limits_{j\to +\infty}\mc K_{\eps_j}(|v_1-v_2|,\gamma^{k-1}_{\eps_j})=\mc K(|v_1-v_2|,\gamma^{k-1}).
 \end{align*} 
Hence, letting $j\to +\infty$ in \eqref{eqmin} we conclude.
\end{proof}

\section{Proof of Theorem~\ref{mainthm}}\label{sec:proof}
We now consider the piecewise constant interpolants $\bm u^\tau$, $\gamma^\tau$ of the discrete displacements and the discrete history slip found in Corollary~\ref{corconv}, namely
\begin{equation}\label{interpolants}
	\begin{cases}
		\bm u^\tau(t):=\bm u^k,& \gamma^\tau(t):=\gamma^k,\qquad\text{if }t\in[t^k,t^{k+1}),\\
		\bm u^\tau(T):=\bm u^{ T/\tau},& \gamma^\tau(T):=\gamma^{ T/\tau}.
	\end{cases}
\end{equation}
It is also convenient to introduce, in an analogous way, the piecewise constant interpolants $w^\tau$ of the external loading
\begin{equation*}
	\begin{cases}
		 w^\tau(t):=w(t^k),\qquad\text{if }t\in[t^k,t^{k+1}),\\
		w^n(T):=w(T).
	\end{cases}
\end{equation*}
Furthermore, for a given $t\in [0,T]$, we define
\begin{equation*}
	t^\tau:=\max\{t^k:\,t^k\le t\}.
\end{equation*}

\begin{prop}\label{propdiscrineq}
	Assume \eqref{Omega}, \eqref{externalloading} and \eqref{initial}. Then for every $t\in [0,T]$ and for all $\tau>0$ (such that $T/\tau\in\N$) the following discrete energy inequality holds true:
	\begin{equation}\label{eq:discrineq}
		\mc F(t^\tau,\bm u^\tau(t),\gamma^\tau(t))\le \mc F(0,\bm u^0,|u_1^0-u^0_2|)+\mc W^\tau(t)+R^\tau,
	\end{equation}
	where $\displaystyle\mc W^\tau(t)=\int_{0}^t\int_{\Omega}\sum_{i=1}^2\mathbb{C}_ie(u^\tau_i(s)):e(\dot{w}(s))\d x\d s$, and $R^\tau\ge 0$ is a (uniform) remainder satisfying $\lim\limits_{\tau\to 0}R^\tau=0$.
\end{prop}
\begin{proof}
	Let us fix $k\in\{1,\dots T/\tau\}$ and for $j=1,\dots k$ we use as a competitor for $\bm u^j$ in \eqref{discralg2}$_1$ the function $\bm u^{j-1}+\bm w(t^j)-\bm w(t^{j-1})$, where we set $\bm w(t)=(w(t),w(t))$, obtaining
	\begin{equation*}
		\mc E(\bm u^j)+\mc K(|u_1^j-u_2^j|,\gamma^{j-1})\le \mc E(\bm u^{j-1}+\bm w(t^j)-\bm w(t^{j-1}))+\mc K(|u_1^{j-1}-u_2^{j-1}|,\gamma^{j-1}).
	\end{equation*}
We now observe that by $(v)$ in Proposition~\ref{propPhi} one has $\mc K(|u_1^j-u_2^j|,\gamma^{j-1})=\mc K(|u_1^j-u_2^j|,\gamma^{j})$. By plugging this fact in the above inequality, after subtracting $\mc E(\bm u^{j-1})$ on both sides and summing from $j=1$ to $k$ we deduce
\begin{align*}
	&\mc E(\bm u^k)+\mc K(|u_1^k-u_2^k|,\gamma^{k})-\mc E(\bm u^0)-\mc K(|u_1^0-u_2^0|,|u_1^0-u_2^0|)\\
	\le&\sum_{j=1}^k
\int_{t^{j-1}}^{t^j}\frac{\d}{\d s}\mc E(\bm u^{j-1}+\bm w(s)-\bm w(t^{j-1}))\d s.
\end{align*}
By computing the above time derivative using the expression \eqref{elasticenergy} we thus obtain
\begin{align*}
	&\mc F(t^k,\bm u^k,\gamma^k)- \mc F(0,\bm u^0,|u^0_1-u^0_2|)\\
	\le&\sum_{j=1}^k
	\int_{t^{j-1}}^{t^j}\int_{\Omega}\left(\sum_{i=1}^2\mathbb{C}_ie(u^{j-1}_i+w(s)-w(t^{j-1}))\right):e(\dot{w}(s))\d x\d s.
\end{align*}
We now fix $t\in [0,T]$ and we rewrite the above inequality in terms of the piecewise constant interpolants. Simple manipulations yield
\begin{align*}
	&\mc F(t^\tau,\bm u^\tau(t),\gamma^\tau(t))- \mc F(0,\bm u^0,|u_1^0-u^0_2|)\\
	\le&\mc W^\tau(t) +(\mc W^\tau(t^\tau)-\mc W^\tau(t))+\int_{0}^{t^\tau}\int_{\Omega}(\mathbb{C}_1+\mathbb{C}_2)e(w(s)-w^\tau(s)):e(\dot{w}(s))\d x\d s\\
	\le & \mc W^\tau(t)+\underbrace{\sup\limits_{\theta\in[0,T]}|\mc W^\tau(\theta^\tau)-\mc W^\tau(\theta)|+ C\int_{0}^T\Vert w(s)-w^\tau(s)\Vert_{H^{1}(\Omega)}\Vert\dot{w}(s)\Vert_{H^{1}(\Omega)}\d s}_{=:R^\tau},
\end{align*}
where in the last inequality we used \ref{hyp:C1}.

We now conclude the proof by showing that $\lim\limits_{\tau\to 0}R^\tau=0$. Indeed by means of \ref{hyp:C1} and \eqref{againbounds} one can bound it by
\begin{align*}
	R^\tau\le C\left( \sup\limits_{\theta\in[0,T]}\int_{\theta^\tau}^{\theta}\Vert\dot{w}(s)\Vert_{H^{1}(\Omega)}\d s+\Vert\dot{w}\Vert_{L^1(0,T;H^{1}(\Omega))}\sup\limits_{\theta\in[0,T]} \Vert w(\theta)-w^\tau(\theta)\Vert_{H^{1}(\Omega)} \right),
\end{align*}
and the right-hand side vanishes by assumption \eqref{externalloading}.
\end{proof}

\subsection{Extraction of convergent subsequences}
	As a consequence of the uniform bounds \eqref{againbounds}, in this section we extract convergent subsequences of the piecewise constant interpolants $(\bm u^\tau,\gamma^\tau)$ defined in \eqref{interpolants}. For the displacements we will employ Banach-Alaoglu Theorem, while for the history slip we will need the following version of Helly's Selection Theorem.
\begin{lemma}\label{helly}
	Let $\{f_n\}_{n\in\N}$ be a sequence of nondecreasing functions from $[0,T]$ to $C^0(\Omega)$ such that:
	\begin{itemize}
		\item[$\bullet$] the families $\{f_n(0)\}_{n\in\N}$ and $\{f_n(T)\}_{n\in\N}$ are locally equibounded;
		\item[$\bullet$] the family $\{f_n(t)\}_{n\in\N}$ is locally equicontinuous, uniformly with respect to $t\in [0,T]$.
	\end{itemize}
Then there exists a subsequence (not relabelled) and a nondecreasing function $f\colon [0,T]\to C^0(\Omega)$ such that $f_n(t)$ converges locally uniformly to $f(t)$ as $n\to +\infty$ for every $t\in [0,T]$.\\
Moreover, for every $t\in [0,T]$ the right and left limits $f^\pm (t)$, which are pointwise well defined by monotonicity, actually belong to $C^0(\Omega)$ and there holds\begin{equation}\label{unifpm}
		f^\pm(t)=\lim\limits_{h\to 0^\pm}f(t+h),\qquad\text{locally uniformly in $\Omega$.}
\end{equation}
\end{lemma}
\begin{proof}
	We consider a countable open exhaustion $\{\Omega_k\}_{k\in \N}$ of $\Omega$, and we observe that by a diagonal argument it is enough to prove the result in a fixed set $\Omega_k$.
	
 	We take a countable and dense set $D\subseteq[0,T]$ containing $0$ and $T$ and by using Ascoli--Arzel\'a theorem and a diagonal argument we can extract a subsequence 
 	(not relabelled) and a function $f$ from $D$ to $C^0(\overline{\Omega_k})$ such that $f_n(t)$ converges uniformly in $\overline{\Omega_k}$ to $f(t)$ for every $t\in D$. Since each $f_n$ is non-decreasing, then 
 	$f$ is non-decreasing on $D$ as well.\par
 	For every $(t,x)\in[0,T]\times \overline{\Omega_k}$ we now define
 	\begin{equation}\label{limrl}
 		f^+(t,x):=\inf\limits_{s\ge t,\,s\in D}f(s,x),\quad f^-(t,x):=\sup\limits_{s\le t,\,s\in D}f(s,x),
 	\end{equation}
 	and we easily observe that
 	\begin{itemize}
 		\item[(a)] $f^-(t)=f(t)=f^+(t)$ for every $t\in D$;
 		\item[(b)] $f^-(t)\le f^+(t)$ for every $t\in [0,T]$;
 		\item[(c)] if $0\le s<t\le T$, then $f^+(s)\le f^-(t)$.
 	\end{itemize}
 	Since the family $\{f_n(t)\}_{n\in\N}$ is equicontinuous in $\overline{\Omega_k}$ uniformly with respect to time we obtain that the limit family $\{f(t)\}_{t\in D}$ is equicontinuous in $\overline{\Omega_k}$ as well.
 	This actually ensures that \eqref{limrl} can be improved in the following way:
 	\begin{equation}\label{unifrl}
 		f^+(t)=\lim\limits_{s\searrow t,\,s\in D}f(s),\quad f^-(t)=\lim\limits_{s\nearrow t,\,s\in D}f(s),\quad \text{ uniformly in }\overline{\Omega_k}.
 	\end{equation}
 	In particular, for every $t\in [0,T]$ the functions $f^+(t)$ and $f^-(t)$ are continuous in $\overline{\Omega_k}$.\par
 	We now introduce the set $E:=\{t\in [0,T]:\, f^+(t)=f^-(t)\}$ and for every $t\in E$ we define $f(t):=f^+(t)=f^-(t)$. Of course, by (a), the set $D$ is contained in
 	$E$ and the definition of $f$ agrees with the one we already had on $D$. We now prove that for every $t\in E$ we have $f_n(t)\to f(t)$ uniformly in $\overline{\Omega_k}$.
 	We already know it is true for $t\in D$, so we assume $t\in E\setminus D$. We fix two points $s'<t<t'$ such that $s',t'\in D$ and since the original sequence was
 	non-decreasing in time we easily get:
 	\begin{align*}
 		&\quad\,\Vert f_n(t)-f(t)\Vert_{C^0(\overline{\Omega_k})}\le \max\left\{	\Vert f_n(t')-f(t)\Vert_{C^0(\overline{\Omega_k})},	\Vert f_n(s')-f(t)\Vert_{C^0(\overline{\Omega_k})}\right\}\\
 		&\le \max\left\{\Vert f_n(t'){-}f(t')\Vert_{C^0(\overline{\Omega_k})}{+}\Vert f(t'){-}f(t)\Vert_{C^0(\overline{\Omega_k})},	\Vert f_n(s'){-}f(s')\Vert_{C^0(\overline{\Omega_k})}{+}\Vert f(s'){-}f(t)\Vert_{C^0(\overline{\Omega_k})}\right\}.
 	\end{align*}
 	Since $s'$ and $t'$ belong to $D$ we thus infer:
 	\begin{equation*}
 		\limsup\limits_{n\to +\infty}\Vert f_n(t)-f(t)\Vert_{C^0(\overline{\Omega_k})}\le \max\left\{	\Vert f(t')-f(t)\Vert_{C^0(\overline{\Omega_k})},	\Vert f(s')-f(t)\Vert_{C^0(\overline{\Omega_k})}\right\}.
 	\end{equation*}
 	Thanks to \eqref{unifrl} and since $t$ is in $E$, letting $t'\searrow t$ and $s'\nearrow t$ we finally conclude that $f_n(t)$ converges uniformly in $\overline{\Omega_k}$ to $f(t)$ for every $t\in E$.\par
 	Let us now show that the set $E^c=[0,T]\setminus E$ is countable.
 	First of all it is easy to see that $E^c$ coincides with the set $\bigcup_{j\in\N}A_{j}$ where for every $j\in \N$ we define
 	\begin{equation*}
 		A_{j}=\left\{t\in[0,T]:\, \int_{\Omega_k}\Big(f^+(t)-f^-(t)\Big)\d x\ge \frac 1j\right\}.
 	\end{equation*}
 	We conclude if we prove that each $A_{j}$ is finite. So we fix $t_1<t_2<\dots<t_r\in A_{j}$ and thanks to (c) we estimate:
 	\begin{align*}
 		\frac rj\le\sum_{i=1}^{r}\int_{\Omega_k}\Big(f^+(t_i)-f^-(t_i)\Big)\d x\le  \int_{\Omega_k}\Big(f^+(t_r)-f^-(t_1)\Big)\d x\le \Vert f(T)-f(0)\Vert_{L^1(\Omega_k)},
 	\end{align*}
 	thus $r$ is bounded from above and thus $A_j$ is finite.\par
 	So by using again Ascoli--Arzel\'a theorem and a diagonal argument we can extract a further subsequence and a function $f$ from $E^c$ to $C^0(\overline{\Omega_k})$ such that
 	$f_n(t)$ converges uniformly to $f(t)$ for every $t\in E^c$. Since in $E$ we already obtained the result, we conclude by noticing that such $f$ is non-decreasing in
 	the whole $[0,T]$ recalling that the original sequence was non-decreasing in time. Indeed \eqref{unifpm} in $\overline{\Omega_k}$ now easily follows by \eqref{unifrl}.
\end{proof}

\begin{prop}\label{propconvsubs}
	Assume \eqref{Omega}, \eqref{externalloading} and \eqref{initial}. Then there exists a subsequence $\tau_j$ and for all $t\in [0,T]$ there exist a further subsequence $\tau_j(t)$ (possibly depending on time), and functions $\bm u(t)\in (H^1_{D,w(t)}\cap C^{0,\alpha}_{\rm loc}({\Omega};\R^n))^2$ and $\gamma(t)\in C^{0,\alpha}_{\rm loc}({\Omega})$ for any $\alpha\in(0,1)$ such that for all $t\in [0,T]$ there hold:
	\begin{equation}\label{convtau}
		\begin{aligned}
			&	\bm u^{\tau_j(t)}(t)\xrightharpoonup[j\to+\infty]{H^1(\Omega;\R^n)^2} \bm u(t),\qquad\text{and}\qquad	\bm u^{\tau_j(t)}(t)\xrightarrow[j\to+\infty]{} \bm u(t) \text{ locally uniformly in $\Omega$},\\
			& \gamma^{\tau_j}(t) \xrightarrow[j\to+\infty]{} \gamma(t) \text{ locally uniformly in $\Omega$}.
		\end{aligned}
	\end{equation}
In particular one has $\bm u(0)=\bm u^0$, $\gamma(0)=|u^0_1-u^0_2|$.\\
Moreover the function $\gamma$ is nondecreasing in time, and
\begin{equation}\label{gammasup}
	\gamma(t,x)\ge \sup\limits_{s\in[0,t]}|u_1(s,x)-u_2(s,x)|, \quad\text{for every $(t,x)\in [0,T]\times\Omega$.}
\end{equation}
Furthermore one has $u\in B([0,T];H^1(\Omega;\R^n)^2)$ and for all $\Omega'\subset\subset\Omega$ and any $\alpha\in(0,1)$ there also hold $u\in B([0,T];C^{0,\alpha}(\overline{\Omega'};\R^n))^2)$ and $\gamma\in B([0,T];C^{0,\alpha}(\overline{\Omega'}))$. 
\end{prop}
\begin{proof}
	The existence of convergent subsequences as in \eqref{convtau} follows by Lemma~\ref{helly} for $\gamma$ and by Banach-Alaoglu and Ascoli-Arzel\'a theorems for $\bm u$, thanks to the uniform bounds \eqref{againbounds}. The same bounds also yield $u\in B([0,T];H^1(\Omega;\R^n)^2)$ and $u\in B([0,T];C^{0,\alpha}(\overline{\Omega'};\R^n))^2)$, $\gamma\in B([0,T];C^{0,\alpha}(\overline{\Omega'}))$ for all $\Omega'\subset\subset\Omega$ and $\alpha\in(0,1)$.
	
	 In order to prove \eqref{gammasup}, assume by contradiction that there exists a pair $(\bar t,\bar x)\in[0,T]\times\Omega$ such that:
	\begin{equation*}
		\gamma(\bar t,\bar x)<\sup\limits_{s\in[0,\bar t\,]}|u_1(s,\bar x)-u_2(s,\bar x)|.
	\end{equation*}
	Thus there exists a time $\bar s\in[0,\bar t\,]$ for which $|u_1(\bar s,\bar x)-u_2(\bar s,\bar x)|>\gamma(\bar t,\bar x)$, and so we infer:
	\begin{align*}
		|u_1(\bar s,\bar x)-u_2(\bar s,\bar x)|&>\gamma(\bar t,\bar x)=\lim\limits_{j\to +\infty}\gamma^{\tau_j}(\bar t, \bar x)\ge\lim\limits_{j\to +\infty}\gamma^{\tau_j}(\bar s, \bar x)\\
		&\ge \limsup\limits_{j\to +\infty}|u_1^{\tau_j}(\bar s,\bar x)-u_2^{\tau_j}(\bar s,\bar x)|\\
		&\ge \lim\limits_{j\to +\infty}|u_1^{\tau_j(\bar s)}(\bar s,\bar x)-u_2^{\tau_j(\bar s)}(\bar s,\bar x)|=|u_1(\bar s,\bar x)-u_2(\bar s,\bar x)|,
	\end{align*}
	which is a contradiction.
\end{proof}

 In the next proposition we pass to the limit the discrete energy inequality stated in Proposition~\ref{propdiscrineq}.

\begin{prop}\label{propEB1}
	Assume \eqref{Omega}, \eqref{externalloading} and \eqref{initial}. Then for every $t\in [0,T]$ the limit functions $\bm u$, $\gamma$ obtained in Proposition~\ref{propconvsubs} satisfy the upper energy inequality
	\begin{equation}\label{EIless}
		\mc F(t,\bm u(t),\gamma(t))\le \mc F(0,\bm u^0,|u_1^0-u_2^0|)+\mc W(t),
	\end{equation}
	where the work of external loading $\mc W$ has been introduced in \eqref{work}.
\end{prop}
\begin{proof}
	We fix $t\in [0,T]$. By lower semicontinuity of the elastic energy in the weak topology of $H^1(\Omega)$ we have 
	\begin{equation}\label{eq:lsc1}
		\mc E(\bm u(t))\le \liminf\limits_{j\to +\infty}\mc E(\bm u^{\tau_j(t)}(t)),
	\end{equation}
	while by Dominated Convergence Theorem together with the locally uniform convergence of $\bm u^{\tau_j(t)}(t)$ and $\gamma^{\tau_j(t)}(t)$ we deduce 
	\begin{equation}\label{eq:lsc2}
		\mc K(|u_1(t)-u_2(t)|,\gamma(t))=\lim\limits_{j\to +\infty} \mc K(|u_1^{\tau_j(t)}(t)-u_2^{\tau_j(t)}(t)|,\gamma^{\tau_j(t)}(t)).
	\end{equation}
	Hence we obtain
	\begin{align*}
		\mc F(t,\bm u(t),\gamma(t))&=\mc E(\bm u(t))+\mc K(|u_1(t)-u_2(t)|,\gamma(t))\\
		&\le \liminf\limits_{j\to +\infty}(\mc E(\bm u^{\tau_j(t)}(t))+\mc K(|u_1^{\tau_j(t)}(t)-u_2^{\tau_j(t)}(t)|,\gamma^{\tau_j(t)}(t)))\\
		&=\liminf\limits_{j\to +\infty}\mc F(t^{\tau_j(t)},\bm u^{\tau_j(t)}(t), \gamma^{\tau_j(t)}(t))\\
		&\le  \mc F(0,\bm u^0,|u_1^0-u^0_2|)+ \limsup\limits_{j\to +\infty} \mc W^{\tau_j(t)}(t),
	\end{align*}
where in the last inequality we exploited \eqref{eq:discrineq}.

By arguing exactly as in \cite[Proposition~3.9]{BonCavFredRiva} one can finally prove that
\begin{equation*}
	\limsup\limits_{j\to +\infty} \mc W^{\tau_j(t)}(t)\le \mc W(t),
\end{equation*}
and we conclude.
\end{proof}

We now show that the limit pair $(\bm u,\gamma)$ fulfils the global stability condition \ref{GS}, with $\gamma$ in place of $\delta_h$.

\begin{prop}\label{propGS}
	Assume \eqref{Omega}, \eqref{externalloading}, \eqref{initial} and \eqref{Mininitial}. Then for every $t\in [0,T]$ the limit functions $\bm u$, $\gamma$ obtained in Proposition~\ref{propconvsubs} satisfy the minimality property
	\begin{equation}\label{eq:globmin}
		\mc F(t,\bm u(t),\gamma(t))\le \mc F(t,\bm v,\gamma(t)),\quad \text{ for every }\bm v\in H^1(\Omega;\R^n)^2.
	\end{equation}
\end{prop}
\begin{proof}
	Fix $t\in [0,T]$. We need to prove that
	\begin{equation*}
		\mc E(\bm u(t))+\mc K(|u_1(t)-u_2(t)|,\gamma(t))\le \mc E(\bm v)+\mc K(|v_1-v_2|,\gamma(t)),\quad \text{ for every }\bm v\in (H^1_{D,w(t)})^2.
	\end{equation*}
If $t=0$ the above inequality is true by assumption \eqref{Mininitial}, so we consider $t\in(0,T]$ and we fix $\bm v\in (H^1_{D,w(t)})^2$.

By means of \eqref{eq:lsc1} and \eqref{eq:lsc2} we have
\begin{align*}
	&\mc E(\bm u(t))+\mc K(|u_1(t)-u_2(t)|,\gamma(t))\\
	\le& \liminf\limits_{j\to +\infty}(\mc E(\bm u^{\tau_j(t)}(t))+\mc K(|u_1^{\tau_j(t)}(t)-u_2^{\tau_j(t)}(t)|,\gamma^{\tau_j(t)}(t)))\\
	=&\liminf\limits_{j\to +\infty}(\mc E(\bm u^{\tau_j(t)}(t))+\mc K(|u_1^{\tau_j(t)}(t)-u_2^{\tau_j(t)}(t)|,\gamma^{\tau_j(t)}(t-\tau_j(t))),
\end{align*}
where in the last equality we used $(v)$ in Proposition~\ref{propPhi}.

 We now exploit the minimality property of the discrete interpolants (see \eqref{discralg2}$_1$) in order to continue the above chain of inequalities: by considering as competitor for $\bm u^{\tau_j(t)}(t)$ the function $\bm v+\bm w(t^{\tau_j(t)})-\bm w(t)\in (H^1_{D,w(t^{\tau_j(t)})})^2$ we obtain
 \begin{align*}
 	&\mc E(\bm u(t))+\mc K(|u_1(t)-u_2(t)|,\gamma(t))\\
 	\le& \liminf\limits_{j\to +\infty}(\mc E(\bm v+\bm w(t^{\tau_j(t)})-\bm w(t))+\mc K(|v_1-v_2|,\gamma^{\tau_j(t)}(t-\tau_j(t)))\\
 	\le & \liminf\limits_{j\to +\infty}(\mc E(\bm v+\bm w(t^{\tau_j(t)})-\bm w(t))+\mc K(|v_1-v_2|,\gamma^{\tau_j(t)}(t)))\\
 	=& \mc E(\bm v)+\mc K(|v_1-v_2|,\gamma(t)),
 \end{align*}
where the second inequality follows by $(iii)$ in Proposition~\ref{propPhi} since $\gamma^{\tau_j(t)}(t-\tau_j(t))\le \gamma^{\tau_j(t)}(t)$, while the last equality is a consequence of the strong continuity in $H^1(\Omega)$ of the elastic energy $\mc E$ (indeed $w(t^{\tau_j(t)})$ strongly converges to $w(t)$ in $H^1(\Omega;\R^n)$ as $j\to +\infty$) and again of the Dominated Convergence Theorem for the cohesive energy $\mc K$. So we conclude.
\end{proof}

 As a standard consequence of the global minimality property \eqref{eq:globmin} one can also deduce the opposite energy inequality to \eqref{EIless}. We refer to \cite[Proposition~3.10]{BonCavFredRiva} for the details.

\begin{prop}\label{propEB2}
	Assume \eqref{Omega}, \eqref{externalloading}, \eqref{initial} and \eqref{Mininitial}. Then for every $t\in [0,T]$ the limit functions $\bm u$, $\gamma$ obtained in Proposition~\ref{propconvsubs} satisfy the lower energy inequality
	\begin{equation}\label{EIgreat}
		\mc F(t,\bm u(t),\gamma(t))\ge \mc F(0,\bm u^0,|u_1^0-u_2^0|)+\mc W(t),
	\end{equation}
	where the work of external loading $\mc W$ has been introduced in \eqref{work}.
\end{prop}

\subsection{Characterization of the irreversible function $\gamma$}\label{subsec:gammadelta}

Up to now we have proved that the limit pair $(\bm u,\gamma)$ obtained in Proposition~\ref{propconvsubs} satisfy both the global stability condition \ref{GS} and the energy balance \ref{EB}, but with the fictitious irreversible function $\gamma$ instead of the history slip $\delta_h$ defined in \eqref{historyslip}. The scope of this section is showing that $\gamma$ actually coincides with $\delta_h$. We will employ a time-regularity argument introduced in \cite{BonCavFredRiva}, which will also ensures the remaining properties of $\bm u$ stated in Theorem~\ref{mainthm}.

To this aim, it is convenient to introduce the \lq\lq shifted\rq\rq energy $\overline{\mc F}\colon [0,T]\times (H^1_{D,0})^2\times L^0(\Omega)^+\to [0,+\infty)$ defined as
\begin{equation*}
	\overline{\mc F}(t,\bar{\bm{u}},\gamma):=\mc F(t,\bar{\bm{u}}+\bm{w}(t),\gamma)=\mc E(\bar{\bm{u}}+\bm{w}(t))+\mc K(|\bar{u}_1-\bar{u}_2|,\gamma).
\end{equation*}

By setting $\bar{\bm u}(t):=\bm u(t)-\bm w(t)$, we can rephrase Propositions~\ref{propEB1}, \ref{propGS} and \ref{propEB2} saying that for all $t\in[0,T]$ there hold
\begin{subequations}
	\begin{align}
		&\label{GS'}\bullet\,\overline{\mc F}(t,\bar{\bm{u}}(t),\gamma(t))\le \overline{\mc F}(t,\bar{\bm{v}},\gamma(t)),\qquad\text{for all }\bar{\bm{v}}\in (H^1_{D,0})^2;\\
		&\label{EB'}\bullet \overline{\mc F}(t,\bar{\bm{u}}(t),\gamma(t))=\overline{\mc F}(0,\bar{\bm{u}}(0),\gamma(0))+\int_{0}^{t}\partial_t\overline{\mc F}(r,\bar{\bm{u}}(r),\gamma(r))\d r.
	\end{align}
\end{subequations}

We also observe that for a.e. $t\in [0,T]$ and for all $\bar{\bm{u}}^a,\bar{\bm{u}}^b\in (H^1_{D,0})^2$ and $\gamma^a,\gamma^b\in L^0(\Omega)^+$ one has
\begin{equation}\label{eq:boundwork}
	|\partial_t\overline{\mc F}(t,\bar{\bm{u}}^a,\gamma^a)-\partial_t\overline{\mc F}(t,\bar{\bm{u}}^b,\gamma^b)|\le C\|\dot w(t)\|_{H^1(\Omega)}\|\bar{\bm u}^a-\bar{\bm u}^b\|_{H^1(\Omega)^2}.
\end{equation}

A key property in order to prove the time-regularity of the displacement $\bm u$ will be the (uniform) convexity of the shifted energy $\overline{\mc F}$ with respect to its second entrance. For this reason, assumption \eqref{eq:hyplambda} will be needed.

\begin{lemma}\label{lemma:convexity}
	Assume \eqref{eq:hyplambda}. Then for every $(t,\gamma)\in [0,T]\times L^0(\Omega)^+$ the map $\bar{\bm{u}}\mapsto \overline{\mc F}(t,\bar{\bm{u}},\gamma)$ is uniformly convex in $(H^1_{D,0})^2$, namely there exists $\mu>0$ such that
	\begin{equation}\label{eq:unifconv}
		\overline{\mc F}(t,\theta\bar{\bm{u}}^a+(1-\theta)\bar{\bm{u}}^b,\gamma)\le \theta\overline{\mc F}(t,\bar{\bm{u}}^a,\gamma)+(1-\theta)\overline{\mc F}(t,\bar{\bm{u}}^b,\gamma)-\frac \mu 2\theta(1-\theta)\|\bar{\bm{u}}^a-\bar{\bm{u}}^b\|^2_{H^1(\Omega)^2},
	\end{equation}
	for all $\theta\in [0,1]$, $\bar{\bm{u}}^a,\bar{\bm{u}}^b\in (H^1_{D,0})^2$.
\end{lemma}
\begin{proof}
	We fix $\theta\in [0,1]$ and $\bar{\bm{u}}^a,\bar{\bm{u}}^b\in (H^1_{D,0})^2$. By definition of $\overline{\mc F}$ we have
	\begin{equation}\label{eq:1}
		\overline{\mc F}(t,\theta\bar{\bm{u}}^a+(1-\theta)\bar{\bm{u}}^b,\gamma)=\mc E(\theta(\bar{\bm{u}}^a+\bm{w}(t))+(1-\theta)(\bar{\bm{u}}^b+\bm{w}(t)))+\mc K(|\theta(\bar u^a_1-\bar u^a_2)+(1-\theta)(\bar u^b_1-\bar u^b_2)|,\gamma).
	\end{equation}
	Since the elastic energy $\mc E$ is a quadratic form, by using \ref{hyp:C5} and Korn-Poincaré inequality \eqref{kornpoincare} we deduce
	\begin{equation}\label{eq:2}
		\begin{aligned}
			&\quad\,\mc E(\theta(\bar{\bm{u}}^a+\bm{w}(t))+(1-\theta)(\bar{\bm{u}}^b+\bm{w}(t)))\\
			&\le \theta\mc E(\bar{\bm{u}}^a+\bm{w}(t))+(1-\theta)\mc E(\bar{\bm{u}}^b+\bm{w}(t))-\frac{c_1\wedge c_2}{2}\theta(1-\theta)\|e(\bar{\bm{u}}^a-\bar{\bm{u}}^b)\|^2_{L^2(\Omega)^2}\\
			&\le \theta\mc E(\bar{\bm{u}}^a+\bm{w}(t))+(1-\theta)\mc E(\bar{\bm{u}}^b+\bm{w}(t))-\frac{c_1\wedge c_2}{2 K^2_2}\theta(1-\theta)\|\bar{\bm{u}}^a-\bar{\bm{u}}^b\|^2_{H^1(\Omega)^2}.
		\end{aligned}
	\end{equation}
On the other hand, by using $(ii)$ and $(v)$ in Proposition~\ref{propPhi} we also obtain that the cohesive energy $\mc K$ is $\lambda$-convex (in the sense of $L^2(\Omega)$) with respect to the first entrance, namely
\begin{equation}\label{eq:3}
	\begin{aligned}
		&\quad\,\mc K(|\theta(\bar u^a_1-\bar u^a_2)+(1-\theta)(\bar u^b_1-\bar u^b_2)|,\gamma)\le \mc K(\theta|\bar u^a_1-\bar u^a_2|+(1-\theta)|\bar u^b_1-\bar u^b_2|,\gamma)\\
		&\le \theta\mc K(|\bar u^a_1-\bar u^a_2|,\gamma)+(1-\theta)\mc K(|\bar u^b_1-\bar u^b_2|,\gamma)+\frac \lambda 2\theta(1-\theta)\||\bar u^a_1-\bar u^a_2|-|\bar u^b_1-\bar u^b_2|\|^2_{L^2(\Omega)}\\
		&\le \theta\mc K(|\bar u^a_1-\bar u^a_2|,\gamma)+(1-\theta)\mc K(|\bar u^b_1-\bar u^b_2|,\gamma)+\lambda \theta(1-\theta)\|\bar{\bm{u}}^a-\bar{\bm{u}}^b\|^2_{L^2(\Omega)^2}.
	\end{aligned}
\end{equation}
By plugging \eqref{eq:2} and \eqref{eq:3} into \eqref{eq:1} we finally obtain
\begin{align*}
	&\quad\,\overline{\mc F}(t,\theta\bar{\bm{u}}^a+(1-\theta)\bar{\bm{u}}^b,\gamma)\\
	&\le \theta\overline{\mc F}(t,\bar{\bm{u}}^a,\gamma)+(1-\theta)\overline{\mc F}(t,\bar{\bm{u}}^b,\gamma)-\frac{1}{2}\left(\frac{c_1\wedge c_2}{K_2^2}-2\lambda\right)\theta(1-\theta)\|\bar{\bm{u}}^a-\bar{\bm{u}}^b\|^2_{H^1(\Omega)^2}.
\end{align*}
We now conclude by setting $\mu:=\frac{c_1\wedge c_2}{K_2^2}-2\lambda$, which is positive by \eqref{eq:hyplambda}.
\end{proof}
\begin{cor}\label{cor:regularity}
	Assume \eqref{Omega}, \eqref{externalloading}, \eqref{initial}, \eqref{eq:hyplambda} and \eqref{Mininitial}. Then the limit function $\bm u$ obtained in Proposition~\ref{propconvsubs} belongs to $AC([0,T];H^1(\Omega;\R^n)^2)$. As a consequence the history slip $\delta_h$, defined in \eqref{historyslip}, can be computed as a pointwise supremum and belongs to $C^0([0,T]\times\Omega)$.
	
	Furthermore, the limit function $\gamma$ belongs to $C^0([0,T]\times\Omega)$ as well.
\end{cor}
\begin{proof}
	By combining \eqref{GS'} together with the uniform convexity \eqref{eq:unifconv} it is standard to infer
	\begin{equation*}
		\frac\mu 2 \|\bar{\bm{u}}(s)-\bar{\bm{v}}\|^2_{H^1(\Omega)^2}+\overline{\mc F}(s,\bar{\bm{u}}(s),\gamma(s))\le \overline{\mc F}(s,\bar{\bm{v}},\gamma(s)),\qquad\text{for all }(s,\bar{\bm{v}})\in [0,T]\times (H^1_{D,0})^2.
	\end{equation*}
We now fix two times $0\le s\le t\le T$ and by choosing $\bar{\bm{v}}=\bar{\bm{u}}(t)$ in the previous inequality we obtain
\begin{align*}
	\frac\mu 2 \|\bar{\bm{u}}(t)-\bar{\bm{u}}(s)\|^2_{H^1(\Omega)^2}\le \overline{\mc F}(s,\bar{\bm{u}}(t),\gamma(s))-\overline{\mc F}(s,\bar{\bm{u}}(s),\gamma(s)).
\end{align*}
Exploiting $(iii)$ in Proposition~\ref{propPhi} together with the monotonicity of $\gamma$, and recalling \eqref{EB'} and \eqref{eq:boundwork} we can continue the above inequality, deducing
\begin{align*}
	\frac\mu 2 \|\bar{\bm{u}}(t)-\bar{\bm{u}}(s)\|^2_{H^1(\Omega)^2}
	&\le \int_{s}^{t}|\partial_t\overline{\mc F}(r,\bar{\bm{u}}(r),\gamma(r))-\partial_t\overline{\mc F}(r,\bar{\bm{u}}(t),\gamma(t))|\d r\\
	&\le C\int_{s}^{t}\|\dot w(r)\|_{H^1(\Omega)}\|\bar{\bm u}(t)-\bar{\bm u}(r)\|_{H^1(\Omega)^2}\d r.
\end{align*}
This implies (see for instance \cite[Lemma~5.6]{GidRiv})
\begin{equation*}
	\|{\bm{u}}(t)-{\bm{u}}(s)\|_{H^1(\Omega)^2}\le\|\bar{\bm{u}}(t)-\bar{\bm{u}}(s)\|_{H^1(\Omega)^2}+\sqrt{2}\|w(t)-w(s)\|_{H^1(\Omega)}\le C\int_{s}^{t}\|\dot{w}(r)\|_{H^1({\Omega})}\d r,
\end{equation*}
and from \eqref{externalloading} we thus conclude that $\bm u$ belongs to $AC([0,T];H^1(\Omega;\R^n)^2)$.

Let us now prove the statements regarding $\delta_h$. For the sake of clarity, for all $t\in [0,T]$ we introduce the notation $\delta(t):=|u_1(t)-u_2(t)|$, so that $\delta_h(t)=\essup_{s\in [0,t]}\delta(s)$.

We now fix $\Omega'\subset\subset\Omega$ and $\alpha\in (0,1)$, and we observe that from the $\alpha$-H\"older continuity of $\bm u$ proved in Proposition~\ref{propconvsubs} we easily deduce
\begin{equation}\label{regdelta}
	\delta\in AC([0,T];L^2(\Omega))\cap B([0,T];C^{0,\alpha}(\overline{\Omega'}))\subseteq C^0([0,T];C^0(\overline{\Omega'})),
\end{equation}
where the inclusion follows by Ascoli-Arzel\'a Theorem. By exploiting \eqref{eq:pointsup}, one then deduces that also $\delta_h$ is in $B([0,T];C^{0,\alpha}(\overline{\Omega'}))$; moreover, by monotonicity, $\delta_h$ belongs to $BV([0,T];L^1(\Omega'))$ and so the right- and left- limits $\delta_h^\pm(t)$ are well defined for all $t\in [0,T]$ in the strong topology of $L^1(\Omega')$ (and by Ascoli-Arzelà even in a uniform sense) and they are ordered, i.e. $\delta_h^+(t)\ge \delta_h(t)\ge\delta_h^-(t)$. Since $\delta$ is continuous with value in $L^2(\Omega)$, necessarily there must hold $\delta_h^+(t)=\delta_h^-(t)$, hence $\delta_h$ is in $C^0([0,T];L^1(\Omega'))$. We have thus proved that
\begin{equation*}
	\delta_h\in C^0([0,T];L^1(\Omega'))\cap B([0,T];C^{0,\alpha}(\overline{\Omega'}))\subseteq C^0([0,T];C^0(\overline{\Omega'})),
\end{equation*}
which yields $\delta_h\in C^0([0,T]\times \Omega)$ by the arbitrariness of $\Omega'$.

Since the same argument can be performed also for the pointwise supremum, here denoted by $\overline{\delta_h}$, one also has $\overline{\delta_h}\in C^0([0,T]\times \Omega)$. By elementary properties of the supremum, this fact implies $\overline{\delta_h}\equiv \delta_h$, so we conclude the part about $\delta_h$.

We are only left to prove the continuity of $\gamma$. By \eqref{unifpm} we already know that $\gamma^\pm(t)$ are well-defined for all $t\in [0,T]$ as locally uniform limits in $\Omega$, and by monotonicity there holds $\gamma^+(t)\ge\gamma(t)\ge\gamma^-(t)$. So we need to prove that $\gamma^+(t)=\gamma^-(t)$ in order to conclude.

 Since we already proved that $\bm u$ is in $AC([0,T];H^1(\Omega;\R^n)^2)$, the map $t\mapsto\mc E(\bm u(t))$ turns out to be absolutely continuous in $[0,T]$; thus \ref{EB}, or better \eqref{EIless} and \eqref{EIgreat}, implies that the map $t\mapsto \mc K(|u_1(t)-u_2(t)|,\gamma(t))$ is absolutely continuous as well. In particular, for every $t\in [0,T]$ there holds
\begin{align*}
	\mc K(|u_1(t)-u_2(t)|,\gamma^+(t))&=\lim\limits_{s\to t^+}\mc K(|u_1(s)-u_2(s)|,\gamma(s))=\lim\limits_{s\to t^-}\mc K(|u_1(s)-u_2(s)|,\gamma(s))\\&=\mc K(|u_1(t)-u_2(t)|,\gamma^-(t)).
\end{align*}
Since by $(iii)$ in Proposition~\ref{propPhi} we know that $\Phi(y\cdot)$ is strictly increasing in $[y,+\infty)$, the above equality yields $\gamma^+(t)=\gamma^-(t)$ and we conclude. 
\end{proof}

In order to prove equality between $\gamma$ and $\delta_h$ we will need the following technical lemma.
\begin{lemma}\label{lemma:limit}
	Assume \eqref{Omega}, \eqref{externalloading}, \eqref{initial}, \eqref{eq:hyplambda} and \eqref{Mininitial}. Then the limit pair $(\bm u,\gamma)$ obtained in Proposition~\ref{propconvsubs} satisfies
	\begin{equation*}
		\lim\limits_{h\to 0}\int_{\Omega}\frac{\Phi(|u_1(t)-u_2(t)|,\gamma(t+h))-\Phi(|u_1(t)-u_2(t)|,\gamma(t))}{h}\d x=0,\qquad\text{for a.e. }t\in [0,T].
	\end{equation*}
\end{lemma}
\begin{proof}
	We adopt the same notation used in the proof of Corollary~\ref{cor:regularity}. Exploiting \eqref{regdelta}, by differentiating \ref{EB} we deduce that for a.e. $t\in [0,T]$ there holds
	\begin{align*}
		0&=\frac{\d}{\d t}\Big(\mc E(\bm u(t))-\mc W(t)+\mc K(\delta(t),\gamma(t))\Big)\\
		&=\sum_{i=1}^{2}\int_{\Omega}\mathbb{C}_i e(u_i(t)):e(\dot{u}_i(t)-\dot{w}(t))\d x+\lim\limits_{h\to 0}\int_{\Omega}\frac{\Phi(\delta(t+h),\gamma(t+h))-\Phi(\delta(t),\gamma(t))}{h}\d x\\
		&=\sum_{i=1}^{2}\int_{\Omega}\mathbb{C}_i e(u_i(t)):e(\dot{u}_i(t)-\dot{w}(t))\d x+\int_{\{\delta(t)>0\}}\partial_y\Phi(\delta(t),\gamma(t))\dot{\delta}(t)\d x\\
		&\quad +\lim\limits_{h\to 0}\int_{\Omega}\frac{\Phi(\delta(t),\gamma(t+h))-\Phi(\delta(t),\gamma(t))}{h}\d x,
	\end{align*}
where we took advantage of the continuity of $\partial_y\Phi$ in $[0,+\infty)^2\setminus\{(0,0)\}$ and we used the fact that $\partial_y\Phi(0,z)=0$ if $z>0$ (see $(ii)$ in Proposition~\ref{propPhi}).

We conclude if we show that the first line in the last equality above is equal to $0$. To this aim, by taking $\bm u(t)+h\bm\varphi$ as variations in \eqref{eq:globmin} and arguing similarly to Proposition~\ref{propeleq}, for all $\bm\varphi\in (H^1_{D,0})^2$ we deduce
\begin{align*}
	0&\le \sum_{i=1}^{2}\int_{\Omega}\mathbb{C}_i e(u_i(t)):e({\varphi}_i)\d x+\lim\limits_{h\to 0^+}\int_{\Omega}\frac{\Phi(|u_1(t){-}u_2(t)+h(\varphi_1{-}\varphi_2)|,\gamma(t))-\Phi(\delta(t),\gamma(t))}{h}\d x\\
	&=\sum_{i=1}^{2}\int_{\Omega}\mathbb{C}_i e(u_i(t)):e({\varphi}_i)\d x+\int_{\{\delta(t)>0\}}\partial_y\Phi(\delta(t),\gamma(t))\dir(u_1(t)-u_2(t))\cdot(\varphi_1-\varphi_2)\d x\\
	&\quad+\psi'(0)\int_{\{\gamma(t)=0\}}|\varphi_1-\varphi_2|\d x,
\end{align*}
	whence
	\begin{align*}
		&\qquad\left|\sum_{i=1}^{2}\int_{\Omega}\mathbb{C}_i e(u_i(t)):e({\varphi}_i)\d x+\int_{\{\delta(t)>0\}}\partial_y\Phi(\delta(t),\gamma(t))\dir(u_1(t)-u_2(t))\cdot(\varphi_1-\varphi_2)\d x\right|\\
		&\le \psi'(0)\int_{\{\gamma(t)=0\}}|\varphi_1-\varphi_2|\d x.	
	\end{align*}
By choosing $\bm{\varphi}=\dot{\bm{u}}(t)-\dot{\bm{w}}(t)\in (H^1_{D,0})^2$ and observing that
\begin{equation*}
	|\varphi_1-\varphi_2|=|\dot{u}_1(t)-\dot{u}_2(t)|=0,\qquad\text{ on the set }\{\gamma(t)=0\},
\end{equation*}
we deduce that
\begin{align*}
	0&= \sum_{i=1}^{2}\int_{\Omega}\mathbb{C}_i e(u_i(t)):e(\dot{u}_i(t){-}\dot{w}(t))\d x+\int_{\{\delta(t)>0\}}\!\!\!\!\!\!\!\!\!\!\!\partial_y\Phi(\delta(t),\gamma(t))\dir(u_1(t){-}u_2(t))\cdot(\dot{u}_1(t){-}\dot{u}_2(t))\d x\\
	&=\sum_{i=1}^{2}\int_{\Omega}\mathbb{C}_i e(u_i(t)):e(\dot{u}_i(t)-\dot{w}(t))\d x+\int_{\{\delta(t)>0\}}\partial_y\Phi(\delta(t),\gamma(t))\dot{\delta}(t)\d x,
\end{align*}
where we exploited the equality $\dot{\delta}(t)=\dir(u_1(t)-u_2(t))\cdot(\dot{u}_1(t)-\dot{u}_2(t))$, which holds true on the set $\{\delta(t)>0\}$. Thus we conclude.
\end{proof}

\begin{prop}\label{propgammadelta}
	Assume \eqref{Omega}, \eqref{externalloading}, \eqref{initial}, \eqref{eq:hyplambda} and \eqref{Mininitial}. Then for the limit pair $(\bm u,\gamma)$ obtained in Proposition~\ref{propconvsubs} one has that $\gamma$ coincides with $\delta_h$.
\end{prop}
\begin{proof}
	We already know by \eqref{gammasup} that $\gamma\ge \delta_h$ in $[0,T]\times\Omega$; moreover $\gamma(0)=\delta_h(0)=|u^0_1-u^0_2|$, so let us assume by contradiction that there exists a point $(\bar t,\bar x)\in (0,T]\times\Omega$ such that
	\begin{equation}\label{eta}
		\gamma(\bar t, \bar x)>\delta_h(\bar t, \bar x).
	\end{equation}
By continuity of both functions (proved in Corollary~\ref{cor:regularity}) the above inequality can be extended in a suitable neighborhood of $(\bar t,\bar x)$, namely there exists $\eta>0$ such that
\begin{equation*}
	\gamma(t,x)>\delta_h(t,x)\ge\delta(t,x),\qquad\text{ for all }(t,x)\in [\bar t-\eta,\bar t\,]\times \overline{B_\eta(\bar x)},
\end{equation*}
where again we set $\delta(t)=|u_1(t)-u_2(t)|$.

Since $\partial_z\Phi$ is continuous and positive on the set $\{z>y\ge 0\}$ (see $(iii)$ in Proposition~\ref{propPhi}) we thus infer the existence of a constant $c_\eta>0$ such that
\begin{equation}\label{eq:lipbelow}
	\Phi(\delta(s,x),\gamma(t,x))-\Phi(\delta(s,x),\gamma(s,x))\ge c_\eta(\gamma(t,x)-\gamma(s,x)),
\end{equation}
for all $\bar t-\eta\le s\le t\le \bar t$ and $x\in  \overline{B_\eta(\bar x)}$.

We now recall that by \ref{EB} (see \eqref{EIless} and \eqref{EIgreat}) the map $t\mapsto \mc K(\delta(t),\gamma(t))$ is absolutely continuous in $[0,T]$, so for all $0\le s\le t\le T$ we obtain
\begin{equation}\label{eq:AC}
\begin{aligned}
	\int_{\Omega}(\Phi(\delta(s),\gamma(t))-\Phi(\delta(s),\gamma(s)))\d x &\le  \mc K(\delta(t),\gamma(t))-\mc K(\delta(s),\gamma(s))+C\psi'(0)\|\delta(t)-\delta(s)\|_{L^2(\Omega)}\\
	&\le C\int_{s}^{t}f(r)\d r,
\end{aligned}
\end{equation}
for a suitable $f\in L^1(0,T)^+$. In the first inequality above we employed $(ii)$ in Proposition~\ref{propPhi}, while we used \eqref{regdelta} in the second one.

By combining \eqref{eq:lipbelow} with \eqref{eq:AC} we finally obtain
\begin{equation*}
	C\int_{s}^{t}f(r)\d r\ge c_\eta\int_{B_\eta(\bar x)}(\gamma(t)-\gamma(s))\d x=c_\eta\|\gamma(t)-\gamma(s)\|_{L^1(B_\eta(\bar x))},\quad\text{for all }\bar t-\eta\le s\le t\le \bar t,
\end{equation*}
and so $\gamma$ belongs to the space $AC([\bar t-\eta,\bar t\,];L^1(B_\eta(\bar x)))$.

We now use Lemma~\ref{lemma:limit}, obtaining for a.e. $t\in [\bar t-\eta,\bar t\,]$
\begin{equation*}
	0=\lim\limits_{h\to 0}\int_{\Omega}\frac{\Phi(\delta(t),\gamma(t+h))-\Phi(\delta(t),\gamma(t))}{h}\d x\ge \limsup\limits_{h\to 0} c_\eta \int_{\Omega}\frac{\gamma(t+h)-\gamma(t)}{h}\ge 0,
\end{equation*}
whence $\gamma$ is strongly differentiable in $L^1(B_\eta(\bar x))$ and $\dot{\gamma}(t)=0$ almost everywhere in $[\bar t-\eta,\bar t\,]$. This implies that $\gamma(\bar t)=\gamma(\bar t-\eta)+\int_{\bar t-\eta}^{\bar t}\dot{\gamma}(r)\d r=\gamma(\bar t-\eta)$ as an equality in $L^1(B_\eta(\bar x))$. Since $\gamma$ is continuous we thus obtain $\gamma(\bar t,\bar x)=\gamma(\bar t-\eta, \bar x)$.

As a consequence, by using monotonicity of $\delta_h$, inequality \eqref{eta} is still true with $\bar t-\eta$ in place of $\bar t$; by repeating the previous argument we hence deduce that $\gamma(\bar t,\bar x)=\gamma(0,\bar x)$. This allows us to conclude, indeed it yields
\begin{align*}
	|u_1^0(\bar x)-u_2^0(\bar x)|=\gamma(0,\bar x)=\gamma(\bar t,\bar x)>\delta_h(\bar t, \bar x)\ge \delta_h(0,\bar x)=|u_1^0(\bar x)-u_2^0(\bar x)|,
\end{align*}
	which is a contradiction.
\end{proof}

	By collecting the results obtained in Propositions~\ref{propconvsubs}, \ref{propEB1}, \ref{propGS}, \ref{propEB2}, \ref{propgammadelta}  and Corollary~\ref{cor:regularity}, we finally conclude the proof of Theorem~\ref{mainthm}.
	
\section{Gradient damage model}\label{sec:damage}

In this last section we enhance the cohesive interface model by considering damageable elastic materials, namely we allow the two layers to undergo a damage process. The literature on damage models in elasticity is wide, we refer for instance to \cite{MielkRoub, Thom} and to \cite[Section 4.3.2]{MielkRoubbook} regarding the mathematical community, and to \cite{AleFredd1d, AleFredd2d, Nejar, PhamMarigo, Wu} for a more engineering point of view. Here, we focus on the so-called gradient damage models, characterized by the presence of a diffusive energy term depending exactly on the gradient of the damage.

In order to incorporate the damaging process in the model, we introduce the additional variable $\bm \alpha(t)=(\alpha_1(t),\alpha_2(t))$ evolving in time. The value $\alpha_i(t,x)\in [0,1]$ represents the amount of damage at time $t\in [0,T]$ of the point $x\in \Omega$ of the $i$-th layer: the value $0$ means that the material is completely sound, while the value $1$ corresponds to a fully damaged state. Since we do not allow healing of the elastic composite, the variables $\alpha_i$ are assumed to be nondecreasing with respect to time.

In the current framework, the total energy of the system can be described by the functional $\mc G\colon [0,T]\times H^1(\Omega;\R^n)^2\times W^{1,r}(\Omega)^2\times L^0(\Omega)^+\to [0,+\infty]$ which reads as
\begin{equation*}
	\mc G(t,\bm u,\bm\alpha,\gamma):=\begin{cases}
		\widetilde{\mc E}(\bm u,\bm\alpha)+\mc D(\bm\alpha)+\mc K(|u_1-u_2|,\gamma),&\text{if }(\bm u,\alpha)\in (H^1_{D,w(t)})^2\times W^{1,r}(\Omega;[0,1])^2 ,\\
		+\infty,&\text{otherwise,}
	\end{cases}
\end{equation*}
where the elastic energy, which now depends also on the damage variable, is given by
\begin{equation}\label{elasticdamage}
	\widetilde{\mc E}(\bm u,\bm\alpha)=\sum_{i=1}^{2}\frac 12 \int_{\Omega} \C_i(x,\alpha_i(x)) e(u_i(x)):e(u_i(x))\d x,
\end{equation}
while the internal and dissipated energy energy related to damage is defined by
\begin{equation}\label{damage}
	\mc D(\bm\alpha)=\sum_{i=1}^{2} \int_{\Omega} \left(w_i(\alpha_i(x))+\frac 1r|\nabla\alpha_i(x)|^r\right)\d x.
\end{equation}

We assume that the two elasticity tensors $\mathbb{C}_i\colon \Omega\times[0,1]\to\R^{n\times n\times n\times n}$ still satisfy conditions \ref{hyp:C1}-\ref{hyp:C5}, meaning that now they hold true for all $(x,\alpha)\in\Omega\times [0,1]$. Moreover we require
\begin{equation}\label{hypdamage}
	w_i\in C^0([0,1])^+,\qquad\text{and}\qquad r>n.
\end{equation}
The high integrability condition $r>n$, which affects the (gradient of the) damage variable $\bm\alpha$, is needed to ensure enough spatial regularity, crucial for our arguments.

In this setting we are able to prove the following result.
\begin{thm}\label{thm2}
	In addition to the hypotheses of Theorem~\ref{mainthm} (except for \eqref{eq:hyplambda}), assume \eqref{hypdamage} and let the initial damage variable $\bm\alpha^0$ belong to $W^{1,r}(\Omega;[0,1])^2$. Then there exists a triple $(\bm u,\bm\alpha,\gamma)\in B([0,T];H^1(\Omega;\R^n)^2\cap C^{0,\alpha}(\overline{\Omega'};\R^n)^2)\times B([0,T];W^{1,r}(\Omega)^2)\times B([0,T];C^{0,\alpha}(\overline{\Omega'})) $ for every $\alpha\in(0,1)$ and every $\Omega'\subset\subset\Omega$ which attains the initial conditions $\bm u(0)=\bm u^0$, $\bm \alpha(0)=\bm \alpha^0$, $\gamma(0)=|u^0_1-u_2^0|$ and satifies the following properties for all $t\in [0,T]$:
	\begin{enumerate}[label=\textup{(IR)}]
		\item the functions $\alpha_i$ and $\gamma$ are nondecreasing in time and $\gamma(t)\ge|u_1(t)-u_2(t)|$;
	\end{enumerate}
	\begin{enumerate}[label=\textup{(GS$\gamma$)}]
		\item \label{GSg} $\mc G(t,\bm u(t),\bm\alpha(t),\gamma(t))\le \mc G(t,\bm v,\bm \beta,\gamma(t)),\quad \text{ for every }(\bm v,\bm\beta)\in H^1(\Omega;\R^n)^2\times W^{1,r}(\Omega)^2$ such that $\beta_i\ge\alpha_i(t)$;
	\end{enumerate}
	\begin{enumerate}[label=\textup{(EB$\gamma$)}]
		\item \label{EBg} $\displaystyle \mc G(t,\bm u(t),\bm\alpha(t),\gamma(t))=\mc G(0,\bm u^0,\bm\alpha^0,|u_1^0-u_2^0|)+\widetilde{\mc W}(t);$
	\end{enumerate}
where now the work of the external forces takes the form
\begin{equation*}
	\widetilde{\mc W}(t)=\int_{0}^{t}\int_{\Omega}\sum_{i=1}^2\C_i(x,\alpha_i(s,x)) e(u_i(s,x)):e(\dot{w}(s,x))\d x\d s.
\end{equation*}	
\end{thm}

Unlike Theorem~\ref{mainthm}, in the current situation we are not able to get rid of the fictitious variable $\gamma$, showing the equality $\gamma=\delta_h$. Indeed, the argument developed in Section~\ref{subsec:gammadelta}, based on time-continuity, here breaks down due to the lack of uniform convexity of the functional $(\bm u, \bm \alpha)\mapsto \widetilde{\mc E}(\bm u,\bm\alpha)+\mc D(\bm \alpha)$. As shown in \cite[Lemma 5.1 and Corollary 5.4]{Thom}, the latter turns out to be uniformly convex (under some structural assumption on the tensors $\mathbb{C}_i$) if the exponent $r$ belongs to $(1,2]$; on the other hand, as it will be explained later, we need to require $r>n$ in order to apply the key result Theorem~\ref{thmregularity}.

\subsection{Proof of Theorem~\ref{thm2}}
Theorem~\ref{thm2} can be proved following the same steps presented in Sections~\ref{sec:regularizedenergy} and \ref{sec:proof}. We now sketch the whole argument highlighting the changes which the presence of the damage variable induces. We also refer to \cite{BonCavFredRiva} for the details in the one-dimensional setting.

We first consider the regularized cohesive energy $\mc K_\eps$ introduced in Section~\ref{sec:regularizedenergy} and, analogously to Section~\ref{subsec:minmov}, we perform the following Minimizing Movements scheme
\begin{equation}\label{mm3}
	\begin{cases}\displaystyle
		(\bm u^k_\eps,\alpha^k_\eps)\in\argmin_{\substack{{\bm{v}\in H^1(\Omega;\R^n)^2,}\\{\bm\beta\in W^{1,r}(\Omega)^2\text{ s.t. }\beta_i\ge (\alpha_i)^{k-1}_\eps}}}\mc G_\eps(t^k,\bm{v},\bm\beta, \gamma^{k-1}_\eps),\\
		\gamma^k_\eps:=\gamma^{k-1}_\eps\vee|(u_1)^k_\eps-(u_2)^k_\eps|,	\\
			\bm u^0_\eps:=\bm u^0,\quad \bm \alpha^0_\eps:=\bm \alpha^0,\quad \gamma^0_\eps:=|u^0_1-u^0_2|,		
	\end{cases}
\end{equation}
where the constraint $\beta_i\ge (\alpha_i)^{k-1}_\eps$ is needed to enforce irreversibility of the damage variables. We point out that the existence of minimizers again follows by the direct method of the Calculus of Variations: coercivity in the weak topology of $H^1(\Omega;\R^n)\times W^{1,r}(\Omega)^2$ comes from the explicit expressions \eqref{elasticdamage} and \eqref{damage}, while lower semicontinuity, which is nontrivial only for the elastic term $\widetilde{\mc E}(\bm u,\bm \alpha)$, follows by means of the Ioffe-Olach Theorem (see \cite[Theorem~{2.3.1}]{Buttazzo}).

By arguing as in Proposition~\ref{prop:boundH1}, but testing against the pair $(\bm w(t^k),\bm 1)$, one deduces the following uniform bounds:
\begin{equation}\label{boundWr}
	\max\limits_{k=0,\dots,T/\tau}\Vert \bm u^k_\eps\Vert_{H^{1}(\Omega)^2}\le C,\qquad \max\limits_{k=0,\dots,T/\tau}\Vert \bm \alpha^k_\eps\Vert_{W^{1,r}(\Omega)^2}\le C.
\end{equation}

By computing the Euler-Lagrange equations of the functional $\bm u\mapsto \mc G_\eps(t^k,\bm u,\bm\alpha^k_\eps,\gamma^{k-1}_\eps)$, it turns out that the discrete displacements $\bm u^k_\eps$ are weak solutions of 
\begin{equation*}
	\begin{cases}
		-\div(\mathbb{C}_1((\alpha_1)^k_\eps) e(u_1))=-\partial_y\Phi_\eps(|u_1-u_2|,\gamma^{k-1}_\eps)\dir(u_1-u_2),&\text{in }\Omega,\\
		-\div(\mathbb{C}_2((\alpha_2)^k_\eps) e(u_2))=\partial_y\Phi_\eps(|u_1-u_2|,\gamma^{k-1}_\eps)\dir(u_1-u_2),&\text{in }\Omega.
	\end{cases}
\end{equation*} 
From \eqref{boundWr}, since $r>n$, Sobolev Embedding Theorem yields that the damage variables $\bm\alpha^k_\eps$ are uniformly continuous in $\Omega$ with moduli of continuity independent of $\eps$ and $\tau$ (they are uniformly $\alpha$-H\"older for a suitable value of $\alpha\in(0,1)$).

This implies that we are still in a position to apply Theorem~\ref{thmregularity} and, arguing as in Proposition~\ref{unifgamma}, we also deduce the additional bounds
\begin{equation*}
	\max\limits_{k=0,\dots, T/\tau}\Vert  \bm u^k_\eps\Vert_{\rm{C^{0,\alpha}}(\overline\Omega')^2}\le C,\qquad\max\limits_{k=0,\dots, T/\tau}\Vert  \gamma^k_\eps\Vert_{\rm{C^{0,\alpha}}(\overline\Omega')}\le C,
\end{equation*}
 where $\alpha\in (0,1)$ and $\Omega'\subset\subset \Omega$ are arbitrary. As a consequence, we deduce the same results of Corollary~\ref{corconv} with in addition
 \begin{equation*}
 	\bm\alpha^k_{\eps_j}\xrightharpoonup[j\to +\infty]{W^{1,r}(\Omega)^2}\bm \alpha^k,
 \end{equation*}
 and with \eqref{discralg2} obviously replaced with the non-regularized version of \eqref{mm3}.
 
 We then consider the piecewise constant interpolants $\bm u^\tau$, $\bm\alpha^\tau$ and $\gamma^\tau$ as in \eqref{interpolants}. We now observe that the following discrete energy inequality 
 \begin{equation*}
 	\mc G(t^\tau,\bm u^\tau(t), \bm \alpha^\tau(t),\gamma^\tau(t))\le \mc G(0,\bm u^0,\bm\alpha^0,|u_1^0-u^0_2|)+\widetilde{\mc W}^\tau(t)+R^\tau,
 \end{equation*}
 can be proven as in Proposition~\ref{propdiscrineq} by testing against $(\bm u^{j-1}+\bm w(t^j)-\bm w(t^j-1),\bm\alpha^{j-1})$.
 
 As in Proposition~\ref{propconvsubs} we can extract convergent subsequences of the piecewise constant interpolants: additionally, for the damage variable we have
 \begin{equation}\label{convalpha}
 	\bm\alpha^{\tau_j(t)}(t)\xrightharpoonup[j\to +\infty]{W^{1,r}(\Omega)^2}\bm \alpha(t),
 \end{equation}
where $\alpha\in B([0,T];W^{1,r}(\Omega;[0,1])^2)$ is nondecreasing and attains the initial datum $\bm\alpha(0)=\bm\alpha^0$.

The validity of \ref{GSg} for the triple $(\bm u,\bm\alpha,\gamma)$ follows by arguing as in Proposition~\ref{propGS} taking as a competitor for the damage variable the function $\bm\beta^{\tau_j(t)}$ whose components are given by \begin{equation*}
	\beta^{\tau_j(t)}_i:=\Big(\beta_i+\|\alpha^{\tau_j(t)}_i(t)-\alpha_i(t)\|_{C^0(\overline{\Omega})}\Big)\wedge 1.
\end{equation*}
Since $\alpha^{\tau_j(t)}_i(t)-\alpha_i(t)$ vanishes uniformly in $\overline\Omega$ as $j\to +\infty$ due to \eqref{convalpha} combined with Sobolev Embedding Theorem, one can indeed prove that $\bm\beta^{\tau_j(t)}$ strongly converges in the sense of $W^{1,r}(\Omega)^2$ to the arbitrary function $\bm\beta$.

The validity of \ref{EBg} instead  can be proved exactly as in Propositions~\ref{propEB1} and \ref{propEB2}, and so we finally conclude the proof of Theorem~\ref{thm2}.
		\bigskip
		
		\noindent\textbf{Acknowledgements.}
		The author is member of the Gruppo Nazionale per l'Analisi Matematica, la Probabilità e le loro Applicazioni (GNAMPA) of the Istituto Nazionale di Alta Matematica (INdAM) and acknowledges its support through the INdAM-GNAMPA Project 2022 \lq\lq MATERIA: Metodi Analitici nella Trattazione di Evoluzioni Rate-independent e Inerziali e Applicazioni\rq\rq, code CUP\textunderscore E55F22000270001.
		
		The author would also like to thank Francesco Freddi and Flaviana Iurlano for fruitful discussions on the considered model.

		{\small
		
		\vspace{15pt} (Filippo Riva) Universit\`{a} degli Studi di Pavia, Dipartimento di Matematica ``Felice Casorati'', \par
		\textsc{Via Ferrata, 5, 27100, Pavia, Italy}
		\par
		\textit{e-mail address}: \textsf{filippo.riva@unipv.it}
		\par
		\textit{Orcid}: \textsf{https://orcid.org/0000-0002-7855-1262}
		\par
		
	}
	
\end{document}